\documentclass[3p]{elsarticle}

\makeatletter
\def\ps@pprintTitle{%
    \let\@oddhead\@empty
    \let\@evenhead\@empty
    \let\@evenfoot\@oddfoot
    }
\makeatother
\usepackage{tabularx}
\usepackage{amsmath}
\usepackage{amssymb}
\usepackage{amsthm}
\usepackage{mathrsfs}
\usepackage{bm}
\usepackage{graphicx}
\usepackage[final]{hyperref}
\usepackage{enumitem}
\usepackage{array}
\usepackage{multirow}
\usepackage{makecell}
\usepackage[title]{appendix}
\usepackage{empheq}
\usepackage{cleveref}
\usepackage{cases}
\usepackage{textcomp}
\usepackage{gensymb}
\usepackage{subcaption}
\hypersetup{
	colorlinks=true,       
	linkcolor=blue,        
	citecolor=blue,        
	filecolor=magenta,     
	urlcolor=blue         
}
\usepackage{float}
\usepackage{multirow}

\usepackage{algorithm}
\usepackage{algpseudocode}

\usepackage[colorinlistoftodos,prependcaption,textsize=normalsize]{todonotes}
\usepackage{regexpatch}

\usepackage{tcolorbox} 
\usepackage{xcolor} 

\makeatletter
\xpatchcmd{\@todo}{\setkeys{todonotes}{#1}}{\setkeys{todonotes}{inline,#1}}{}{}
\makeatother

\newtheorem{proposition}{Proposition}

\newtheorem{theorem}{Theorem}

\usepackage{tikz}
\usetikzlibrary{calc,patterns}
\tikzstyle{bag} = [align=center]

\begin{document}
\begin{frontmatter}
\title{Derivative-Informed Neural Operator:\\ An Efficient Framework for\\ High-Dimensional Parametric Derivative Learning}
\author[1]{Thomas O'Leary-Roseberry\corref{cor1}}
\ead{tom.olearyroseberry@utexas.edu}
\author[2]{Peng Chen}
\ead{pchen402@gatech.edu}
\author[1]{Umberto Villa}
\ead{uvilla@oden.utexas.edu}
\author[1,3]{Omar Ghattas}
\ead{omar@oden.utexas.edu}

\address[1]{Oden Institute for Computational Sciences and Engineering, The University of Texas at Austin, 201 E. 24th Street, C0200, Austin, TX 78712.}

\address[2]{College of Computing, Georgia Institute of Technology,  801 Atlantic Drive NW, Atlanta, GA, 30332.}

\address[3]{Walker Department of Mechanical Engineering, The University of Texas at Austin, 204 E. Dean Keaton Street, Austin, TX 78712.}

\cortext[cor1]{Corresponding author}

\begin{abstract}
We propose derivative-informed neural operators (DINOs), a general family of neural networks to approximate operators as infinite-dimensional mappings from input function spaces to output function spaces or quantities of interest. After discretizations both inputs and outputs are high-dimensional. We aim to approximate not only the operators with improved accuracy but also their derivatives (Jacobians) with respect to the input function-valued parameter to empower derivative-based algorithms in many applications, e.g., Bayesian inverse problems, optimization under parameter uncertainty, and optimal experimental design. The major difficulties include the computational cost of generating derivative training data and the high dimensionality of the problem leading to large training cost. To address these challenges, we exploit the intrinsic low-dimensionality of the derivatives and develop algorithms for compressing derivative information and efficiently imposing it in neural operator training yielding derivative-informed neural operators. We demonstrate that these advances can significantly reduce the costs of both data generation and training for large classes of problems (e.g., nonlinear steady state parametric PDE maps), making the costs marginal or comparable to the costs without using derivatives, and in particular independent of the discretization dimension of the input and output functions. Moreover, we show that the proposed DINO achieves significantly higher accuracy than neural operators trained without derivative information, for both function approximation and derivative approximation (e.g., Gauss--Newton Hessian), especially when the training data are limited.
\end{abstract}


\end{frontmatter}

\tableofcontents

\section{Introduction}

So-called ``neural operators'' have gained significant attention in recent years due to their ability to approximate high-dimensional parametric maps between function spaces, and have become a major research topic in scientific machine learning (SciML) \cite{BhattacharyaHosseiniKovachkiEtAl2021,FrescaManzoni2022,KovachkiLiLiuEtAl2021,LiKovachkiAzizzadenesheliEtAl2020a,LiKovachkiAzizzadenesheliEtAl2020b,LuJinKarniadakis2019,OLeary-RoseberryVillaChenEtAl22,OLearyRoseberryDuChaudhuriEtAl2021}. By offering an inexpensive approximation of high-dimensional parameter to output maps (denoted by $m \mapsto q$), neural operators are playing an increasingly critical role in making tractable the solution of PDE-constrained ``outer-loop'' problems, i.e. problems that require numerous solutions of parametric PDEs as parameter space is explored. Examples of outer-loop problems include: forward uncertainty propagation, deterministic and Bayesian inverse problems, optimal experimental design, and deterministic and stochastic optimal design and control.

A fundamental issue has not yet been addressed for neural operators: Many algorithms for outer-loop problems require accurate derivatives of functions of the outputs $F(q(m))$ taken with respect to the high-dimensional input parameters, $\nabla_m F(q(m)$. For outer-loop problems formulated as optimization problems, gradients and Hessians are important for accelerating convergence of iterative methods to (local) minimizers. For outer-loop problems that involve sampling parameter space, the availability of gradients and Hessians can greatly accelerate sampling algorithms (e.g. geometric MCMC \cite{BeskosGirolamiLanEtAl17,MartinWilcoxBursteddeetal2012} and derivative-based control variates \cite{AlexanderianPetraStadlerEtAl17,ChenGhattas21,ChenVillaGhattas19}). These derivative quantities can be efficiently computed using the Jacobian of the neural operator outputs with respect to its inputs via the chain rule: $\nabla_m F(q(m)) = \frac{\partial F}{\partial q}\nabla q$. However, errors in approximating the parametric Jacobian $\nabla q$ will be inherited in these computations, leading to inaccurate approximations of derivative quantities. This shortcoming limits the deployment of neural operators in outer-loop algorithms, since it restricts one to the use of slow-converging derivative-free methods, or to inaccurately approximate derivatives in derivative-based methods, perhaps yielding inaccurate derivative based subspaces or convergence to stationary points that do not coincide with the stationary points of the true model.

In this work we design algorithms to efficiently include high-dimensional derivative information in neural operator training to create ``\emph{derivative-informed neural operators}'' (DINOs), which faithfully approximate not only the parametric maps, but also the derivatives of their outputs with respect to the input parameters (i.e. parametric derivatives). The goal of this inclusion is two-fold

\begin{enumerate}
    \item Provide an efficient framework to improve parametric Jacobian accuracy, thus enabling more accurate derivative computations used in outer-loop algorithms.

    \item To improve function accuracy of neural operators by incorporating high-dimensional derivative information of the parametric map into the training. This added derivative information has the effect of better interpolation between training data points by imposing the constraint that the slopes at these points match the true parametric map. This can be thought of as philosophically similar to Hermite polynomial interpolation.

\end{enumerate}

In the setting we are concerned with, the forward model is parameterized by parameters $m \in \mathbb{R}^{d_M}$ with
probability distribution $m \sim \nu$, which are mapped to outputs $q(m) \in \mathbb{R}^{d_Q}$ through an implicit dependence on the solution of state equations for the state variable $u \in \mathbb{R}^{d_U}$: 
\begin{equation}\label{eq:state}
  \underbrace{q(m) = q(u(m))}_\text{QoI mapping} \text{ where } u \text{ depends on } m \text{ through } \underbrace{R(u,m) = 0}_\text{State equations}.
\end{equation}
These parametric maps usually arise via the discretization of formally infinite-dimensional maps on function spaces; the canonical example being the discretization of continuum fields arising in partial differential equations (PDEs), which is our primary focus. The underlying PDEs can be either steady-state or time-dependent; in the latter case the parameters and states can be functions in space-time, discretized as space-time vectors. For these reasons the dimensions $d_M,d_Q,d_U \in \mathbb{N}$ can be very large, though not necessarily indicative of the ``intrinsic dimensionality'' of the forward model and its derivatives, i.e. the number of basis vectors required to adequately capture the input to output map and its derivatives. 

The neural operator $f_w(m) = f(m,w):\mathbb{R}^{d_M}\times \mathbb{R}^{d_W}\rightarrow \mathbb{R}^{d_Q}$, parametrized by both $m$ and the weights $w \in \mathbb{R}^{d_W}$, is a cheap-to-evaluate neural network-based surrogate for the map $m\mapsto q$. The neural operator ``learns'' an approximation of the map by minimizing a data and/or physics informed objective or loss function via nonlinear stochastic optimization with respect to $w$. We focus specifically on neural operators as no other surrogate strategies have proven successful for very high-dimensional and nonlinear maps, however the methodologies proposed herein are applicable to other surrogate methodologies. The development of neural operators has to this point focused on approximating the functional relationship $m \mapsto q$ over the distribution $\nu$. Since approximation of the Jacobian of $f_w(m)$ with respect to $m$ is not addressed directly in this context, it is generally unreasonable to expect derivatives $\nabla f_w$ to be reliable approximations of $\nabla q$, particularly when $m$ is high-dimensional. 

We address the computational challenges associated with learning high-dimensional derivatives, by proposing dimension-independent methods for derivative learning, i.e., where the dominant offline costs measured in PDE solves, and online training costs are independent of the discretization parameters $d_Q,d_U,d_M$. When $m$ and $q$ are vectors, $\nabla q$ and $\nabla_m f_w$ are Jacobian matrices with $d_Q\times d_M$ entries. Directly computing and imposing full derivative information is intractable when $d_Q,d_M$ are large, since the number of (linearized) forward or adjoint PDE solves required to compute $\nabla q$ is the rank of $\nabla q$ that could be as high as $\min\{d_Q,d_M\}$. However, for many PDE problems, parametric derivatives can be represented compactly, i.e. the effective rank of the Jacobian can be shown to be small and independent of the discretization dimension. Jacobian compactness can be due to a number of reasons, including smoothing properties of the PDE solution, sparse observation QoI, and compactness of the prior distribution. This compactness makes representing extrinsically high-dimensional Jacobian information tractable. This property can be proven analytically for certain elliptic, parabolic and diffusive hyperbolic model problems \cite{GhattasWillcox21}, and is empirically observed for many others in model reduction \cite{AlgerChenGhattas20,BashirWillcoxGhattasEtAl08,
  ChenGhattas19a,GhattasWillcox21}, Bayesian inversion
\cite{Bui-ThanhBursteddeGhattasEtAl12, Bui-ThanhGhattas12a,
  Bui-ThanhGhattas12, Bui-ThanhGhattas13a, Bui-ThanhGhattas15,
  Bui-ThanhGhattasMartinEtAl13,ChenGhattas20,ChenVillaGhattas17, ChenWuChenEtAl19, FlathWilcoxAkcelikEtAl11,
  IsaacPetraStadlerEtAl15, MartinWilcoxBursteddeEtAl12,PetraMartinStadlerEtAl14a}, optimization
under uncertainty \cite{AlexanderianPetraStadlerEtAl17,ChenGhattas21, ChenHabermanGhattas21,
  ChenVillaGhattas19}, and
Bayesian optimal experimental design
\cite{AlexanderianPetraStadlerEtAl14, AlexanderianPetraStadlerEtAl16,
CrestelAlexanderianStadlerEtAl17, WuChenGhattas22,WuChenGhattas22b,WuOLearyRoseberryChenEtAl23}. When the dimensionality of the derivative information (denoted by $r$) throughout parameter space is independent of the continuum discretization dimensions $d_M,d_U$, we demonstrate that efficient and scalable methods can be used to compute and include Jacobian training data in neural operator training. 

We begin in Section \ref{section:dino} by formulating our goal: to solve an empirical risk minimization problem that includes the forward model Jacobian in addition to function values. We proceed in Section \ref{section:scalable_computation} by overviewing how derivative information at sample points can be efficiently computed using adjoints and randomized matrix-free methods. This amounts to only $O(r)$ additional linear solves with the same PDE operator at each sample point, which can leverage concurrency and computations already required for the (possibly nonlinear) forward solve, such as direct factorization or preconditioner construction. We then discuss how the dominant Jacobian information can be captured from its low rank (truncated) SVD representation, which can be computed matrix-free using efficient randomized methods \cite{HalkoMartinssonTropp11}. In Section \ref{section:jacobian_svd}, we discuss how the derivative training objective can be efficiently approximated through use of the reduced SVD representation of the Jacobian, which reduces the online memory and computational complexity from $O(d_Q\times d_M)$ to $O(r^2+d_Qr + d_Mr)$ per sample point. We show that matrix subsampling can be used to further reduce these costs when $r$ is large using stochastic approximation of the truncated Jacobian matrix error. While this approach is generally applicable for all neural operators, a major limitation is that it fails to properly penalize the Jacobian error in the orthogonal complements of the truncated SVD singular vectors; while these approaches benefit the function approximation, these errors in the orthogonal complements ultimately pollute predictions of derivative-based quantities.

In Section \ref{section:arch_constraints}, we overcome this issue at the architectural level by suggesting that the most efficient way to learn parametric maps with derivatives is to use architectures that exploit the intrinsic dimensionality of the map via the use of reduced basis neural operators. These architectures--which restrict their nonlinearity to only the coordinates of reduced bases of the inputs and outputs--are ideal for high-dimensional derivative learning, since the reduced basis representation precludes the learning of erroneous Jacobian information in the orthogonal complements of the neural operator's reduced bases. The dominant online training costs (per sample point) associated with the reduced basis neural operator learning scale with the dimensionality of the reduced bases $\bar{r}$, instead of the ambient discretization dimensions $d_M,d_Q$. When the reduced bases can accurately resolve the dominant Jacobian information over parameter space, these architectures are ideally for derivative learning. This simple strategy leads to a powerful approach for learning high-dimensional derivatives, and constitutes the fundamental result of this paper.

In numerical experiments Section \ref{section:numerical_results}, we consider high-dimensional derivative learning for three different parametric nonlinear PDE problems. In our numerical results we compare two reduced basis architectures with one generic encoder-decoder network and consider several different proposed training formulations, which are proposed in this work. Our numerical results demonstrate the following three takeaways. First, the inclusion of derivative information significantly improves the function approximation. Thus, for highly nonlinear problems, derivative information represents an economical means of improving function approximations, since the costs of computing Jacobian information can be made small relative to nonlinear forward model evaluations. Second, the inclusion of derivative information in the loss function for neural operator training is essential for neural operators to have accurate derivatives in high-dimensions. In many cases one can achieve accurate function approximations while simultaneously exhibiting poor approximation of derivative quantities. This empirical finding suggests that caution should be exercised when differentiating high-dimensional neural operators that are not trained using derivative information. Third, the use of reduced basis neural operators reduces costs of computing training data, makes imposition of complete derivative information in the training problem tractable, and allows the neural operator to fully learn the derivatives when the reduced bases capture the left and right singular vectors of true Jacobian over parameter space. 



\subsection{Related works}

The idea of incorporating derivative information in neural network has been suggested before in \cite{CzarneckiOsinderoJaderbergEtAl2017}. In this setting the authors restricted themselves to learning analytic functions where derivatives were easily available, and then supplementing ``synthetic derivatives'' in data-driven problems where there is no notion of a \emph{true} differentiable map, but the inclusion of synthetic derivative information had a regularizing effect and led to superior approximation. The present work differs from this work due to the emphasis on efficient methods for handling high-dimensional maps with continuum limits, where derivatives are known to exist, but learning full derivatives is computationally infeasible since the dimensions of the problem grow as resolution increases. 

A more recent preprint \cite{BigoniMarzoukPrieurEtAl2021} investigates using derivative information for constructing dimension reduced surrogates for scalar valued parametric maps via aligning the gradients of the surrogate with the true gradients of the scalar map. The focus of that paper is on using derivative sensitivity information to reduce the surrogate's input dimensionality, and not directly the training of the surrogates derivatives, as in our work. Additionally the paper does not consider high-dimensional outputs.

There are many other papers on learning parametric PDE maps in a SciML context \cite{BhattacharyaHosseiniKovachkiEtAl2021,FrescaManzoni2022,KovachkiLiLiuEtAl2021,LiKovachkiAzizzadenesheliEtAl2020a,LiKovachkiAzizzadenesheliEtAl2020b,NelsenStuart2020,OLeary-RoseberryVillaChenEtAl22,OLearyRoseberryDuChaudhuriEtAl2021,RaissiPerdikarisKarniadakis2019}. The algorithms proposed in this work are suitable additions for any of these methods to learn high-dimensional derivative information and additionally improve approximations of the maps themselves. As we will demonstrate, derivative learning is particularly efficient for reduced basis neural network architectures \cite{BhattacharyaHosseiniKovachkiEtAl2021,lu2022comprehensive,OLeary-RoseberryVillaChenEtAl22,OLearyRoseberryDuChaudhuriEtAl2021}. This work concerns \emph{parametric derivatives}, not to be confused with \emph{spatial derivatives}, which have been addressed in neural operator construction in various works \cite{LiKovachkiAzizzadenesheliEtAl2020a,YuLuMengEtAl2022}.

\subsection{Contributions}

The central contribution of this paper is the development of novel loss functions and computationally efficient stochastic optimization problem formulations to incorporate derivative information in neural operator learning. Our proposed methods not only enable the use of neural operators in derivative-based outer-loop algorithms, but additionally improve the fidelity of the neural operator in the classical mean squared error loss. This is achieved by the use of dimension reduction algorithms whose dominant costs per sample point are independent of the discretization dimensions $d_Q,d_M$, and instead depend on the intrinsic dimension $r$ of the Jacobian. We begin by reviewing how derivative information can be computed efficiently using matrix-free adjoint methods, which require only $O(r)$ linear solves per sample point to capture Jacobian information. These linear solves are often inexpensive relative the nonlinear forward evaluation of the parametric map, since they involve the same linear operator. In order to make the online memory and computational costs of derivative learning independent of $d_Q\times d_M$, we consider two main approaches delineated below.

\textbf{Efficient formulations of derivative learning for general neural operators}: We consider techniques based on truncated SVD to impose only the dominant Jacobian information in the truncated SVD basis, reducing the computational costs from $O(d_Qd_M)$ to $O(r^2+d_Qr + d_Mr)$ where $r$ is the rank of the Jacobian. We refer to this general approach as truncated Jacobian training. When demonstrate that when $r$ is large, the computational cost can be reduced to $O(k^2+kd_Q + kd_M)$ per sample point, by making use of matrix subsampling techniques that randomly choose $k$ rows and columns of the truncated Jacobian, and prove that this provides an asymptotically equivalent objective function in Proposition \ref{prop:submatrix_svd_convergence}. The techniques based on truncated SVD are highly scalable, independent of the neural network architecture and can significantly improve the approximation of the parametric map by imposing slope constraints in the directions where the map is changing most, but alone are not sufficient to guarantee high parametric Jacobian accuracy. 

\textbf{Reduced basis derivative learning}: We show that reduced basis neural operators can be used to provide an efficient and scalable framework for accurate parametric derivative learning. Regarding this task, the fundamentally important feature of these neural operators is that their reduced bases preclude learning erroneous Jacobian information in the orthogonal complements of their $\bar{r}$ dimensional reduced bases, as we prove in Proposition \ref{prop:jacobian_ridge_orthogonality}. As a consequence of this result, we prove that the associated derivative learning problem is dimension-independent (Theorem \ref{theorem:dim_independent_dino}) for reduced basis neural operators, meaning that computing derivative training data requires only $O(\bar{r})$ linear solves and the online memory and computational costs of derivative learning are $O(\bar{r}^2)$ instead of $O(d_Qd_M)$. This is the major result of this paper and is demonstrably the best approach we considered, as evidenced by numerical experiments. 

We are unaware of any techniques for the efficient training of high-dimensional parametric surrogates (neural operators) on derivative information. We restrict ourselves to learning the first order parametric derivative (i.e. parametric Jacobian), but the ideas can be to be extended to higher derivatives using similar ideas. 

The optimization formulations we propose are general and can be applied to any parametric operator learning task; i.e., they apply to both steady-state and time-dependent problems. However, the efficiency of the proposed methods relies on: (1) the ability to efficiently compute Jacobian information (e.g., via the method of adjoints), and (2) the compressibility of the Jacobian information. Regarding the first issue, while the adjoint method is suitable for a very large class of PDE problems, there are some problems for which it does not work. For example high Reynolds number fluid simulations via direct numerical simulations or large eddy simulations over long time horizons are known to have chaotic adjoints, while Reynolds averaged Navier Stokes models do not have this property and the adjoint method is still suitable \cite{ChenVillaGhattas19}. Concerning the second issue, the approaches we propose in this work rely on the compressibility of high-dimensional Jacobian information; DINO may not be suitable for problems that have high effective rank of the Jacobian. This is typical of hyperbolic problems such as the wave equation \cite{GhattasWillcox21}, advection-dominated transport \cite{FlathWilcoxAkcelikEtAl11}, or high Reynolds number flow. In the current presentation we rely on the decay of Jacobian singular values for compressibility. In the numerical results we focus on steady-state problems that meet these criteria for simplicity. 

\section{Derivative-informed neural operator} \label{section:dino}

We address parametric derivative learning for high-dimensional maps $\mathbb{R}^{d_M}\ni m \mapsto q(m) \in \mathbb{R}^{d_Q}$, with parameter $m$ obeying probability density $\nu$. Let $H^1_\nu = H^1(\mathbb{R}^{d_M},\nu; \mathbb{R}^{d_Q})$ denote a (discrete) Bochner space with the norm given by
\begin{equation}
    \begin{split}
    \|q\|_{H^1_\nu} &= \sqrt{\mathbb{E}_\nu[\|q\|_2^2 + \|\nabla q\|_F^2]} =  \sqrt{\int_{\mathbb{R}^{d_M}} \left(\|q(m) \|^2_2 + \|\nabla q(m)\|_F^2 \right)d\nu(m)},
    \end{split}
\end{equation}
where $\mathbb{E}_\nu$ is the expectation taken with respect to $\nu$, and $\|q(m)\|_2 = \|q(m)\|_{\ell^2(\mathbb{R}^{d_Q})}$ is the vector Euclidean norm of the output $q$ evaluated at $m$ and $\|\nabla q(m)\|_F = \|\nabla q(m)\|_{F(\mathbb{R}^{d_Q\times d_M})}$ is the matrix Frobenius norm of the Jacobian $\nabla q = \nabla_m q =  \frac{\partial q}{\partial m}$ evaluated at $m$. We refer to the square root of the expectation of this term, $\sqrt{\mathbb{E}_\nu[\|\nabla q\|^2_F]}$ as the $H^1_\nu$ semi-norm. The space $H^1_\nu$ is a subset of $L^2_\nu = L^2(\mathbb{R}^{d_M},\nu; \mathbb{R}^{d_Q})$, including all $L^2_\nu$ functions that have finite $H^1_\nu$ semi-norm. The derivative-informed neural operator $f_w = f(\cdot,w) :\mathbb{R}^{d_M}\times \mathbb{R}^{d_W}\rightarrow \mathbb{R}^{d_Q}$ is a neural network parametrized by weights $w\in \mathbb{R}^{d_W}$, and is constructed by attempting to solve the derivative-informed expected risk minimization problem:
\begin{equation}\label{h1_exp_risk}
   \min_w \frac{1}{2} \, \mathbb{E}_\nu \left[\|q - f_w\|_2^2+ \|\nabla q - \nabla f_w\|_F^2\right].
\end{equation}
This minimization problem attempts to train a neural operator that simultaneously approximates the QoI map as well as its parametric Jacobian on average over parameter space. Due to the inability to directly integrate with respect to $\nu$ in practice, the expected risk minimization problem is approximated via samples of the map and its derivatives \linebreak$\{(m_i,q(m_i),\nabla q(m_i))|m_i \sim \nu\}_{i=1}^N$, which leads to the empirical risk minimization problem:
\begin{equation}\label{h1_emp_risk} 
  \min_w \frac{1}{2N}\sum_{i=1}^N \|q(m_i) - f_w(m_i)\|_2^2+\|\nabla q(m_i) - \nabla f_w(m_i)\|_F^2,
\end{equation}
which is solved using gradient-based methods, by differentiating \eqref{h1_emp_risk} with respect to the weights $w$. While it would be ideal to solve the empirical risk minimization problem \eqref{h1_emp_risk} as stated, since it imposes the complete Jacobian information, this becomes computationally infeasible as the dimensions $d_Q,d_M$ grow. This is due to the costs involved in both constructing $\nabla q$ training data, as well as the memory and computational costs associated with evaluating and differentiating with respect to $w$ the $d_Q\times d_M$ Jacobian approximation error term, $\|\nabla q(m_i) - \nabla f_w(m_i)\|_F^2$. In the former, ostensibly $\min\{d_M,d_Q\}$ linearized forward or adjoint solves are required at each training data point to build the Jacobian. In the training phase, computational graphs must allocate $O(N_\text{batch}d_Qd_M)$ memory for batch computation of Jacobian approximation errors, as well as $3-6X$ additional storage used in automatic differentiation \cite{baydin2015automatic}.

In this work we propose efficient methods that can reduce the computation costs associated with both Jacobian training data generation as well as the neural operator training by making use of powerful dimension reduction strategies. High-dimensional parametric maps of interest often locally admit compact representations; they are informed in only low-dimensional subspaces of the inputs and outputs. In such cases derivative information can also be shown to be low rank (i.e. the derivative admits a compact representation). This happens because of the smoothness of the map, and correlation in both the input parameters and the output QoIs. In this case, local linearizations of the map can be efficiently approximated by low rank matrices with rank $r \ll d_Q, d_M$, when $d_Q$ is high-dimensional (e.g. representing a continuum field) or $r = O(d_Q) \ll d_M$, when $d_Q$ is finite dimensional and small. This low-rank property of the Jacobian matrix has been observed and used in inverse problems across many applications
\cite{BeskosGirolamiLanEtAl17,BrennanBigoniZahmEtAl20,Bui-ThanhBursteddeGhattasEtAl12,
Bui-ThanhGhattasMartinEtAl13,Bui-ThanhGhattas14,ChenGhattas21,ChenVillaGhattas17,ChenWuChenEtAl19,CuiLawMarzouk16,FlathWilcoxAkcelikEtAl11,IsaacPetraStadlerEtAl15,ZahmCuiLawEtAl18}. Due to this low-rank property, Jacobian training data can be computed at the cost of $O(r)$ linearized forward or adjoint solves at each sample point, e.g., using randomized SVD \cite{HalkoMartinssonTropp11}, and the online computational and memory costs of training can be made to scale with $r$ instead of the discretization dimensions $d_M,d_Q$ (per sample). 

\subsection{Scalable computation of Jacobian training data} \label{section:scalable_computation}

After samples of $m, u(m)$ are computed, Jacobian training data can be computed matrix-free via differentiation of $q(m) = q(u(m))$ with respect to $m$, implicitly accounting for the state equations $R(u(m),m) = 0$:
\begin{equation} \label{eq:implicit_jacobian}
  \nabla q(u(m)) = -\frac{dq}{du}\left[\frac{\partial R}{\partial u}\right]^{-1}\frac{\partial R}{\partial m}.
\end{equation}
When the Jacobian rank $r$ is small, $\nabla q$ can be efficiently approximated using randomized SVD \cite{HalkoMartinssonTropp11}, where the Jacobian is accessed via its action on only $O(r)$ random vectors. The major computational costs are due to the matrix inverse $\left[\frac{\partial R}{\partial u}\right]^{-1}$, which can be applied efficiently to random vectors from the right or left by solving linearized forward or adjoint PDEs, respectively. In many PDE problems, $\frac{\partial R}{\partial u}$ is sparse and can be efficiently factored using sparse direct solvers, which can amortize the factorization cost by simultaneous application to many right-hand sides, as is done in matrix-free methods to compute the truncated SVD of the Jacobian matrix-free. When the forward model is linear, the matrix $\left[\frac{\partial R}{\partial u}\right]^{-1}$ is already factorized from the state solution. When the forward model is nonlinear, factorizing and applying $\left[\frac{\partial R}{\partial u}\right]^{-1}$ are typically inexpensive relative to the costs of the state solution $m \mapsto u(m)$, and may be readily factorized and available from the last Newton iteration used in the solution of the nonlinear forward model. 

For very high-dimensional problems where sparse direct solvers cannot be used, one can further amortize the costs of preconditioners used to accelerate the solution of the forward model. In general, Jacobian matrices that are of low rank, or compressible in some other way (e.g., as hierarchical matrices) should not be formed explicitly as dense matrices, and can be efficiently compressed using matrix-free methods that require few PDE solves. For highly nonlinear problems (measured in number of nonlinear Newton iterations required for their solution), the marginal costs of Jacobian training data generation can be made small relative to the solution of the nonlinear forward model when the Jacobian is compressible. 



See Appendix \ref{section:train_data_generation} for pseudo code explaining the derivative training data generation process.

\subsection{Approximation of matrix Frobenius norm via truncated SVD} \label{section:jacobian_svd}

At each parameter sample we can decompose a Jacobian with rank $r$ into its truncated SVD:
\begin{subequations}
\begin{equation}\label{eq:qSVD}
    \nabla q \approx U_r \Sigma_r V_r^T, \qquad \|\nabla q - U_r\Sigma_r V_r^T\|_F^2 = \sum_{i=r+1}^{\min\{d_Q,d_M\}} \sigma_i^2.
\end{equation}

The rank $r$ is chosen such that truncated singular values $\sigma_i$ are small or zero for $i >r$. When the truncated singular values are exactly zero the truncated SVD coincides with the reduced SVD. All of the information about the Jacobian is contained in the reduced SVD, since the left and right nullspaces of the Jacobian coincide with the orthogonal complements to the dominant left and right singular vectors ($U_r^\perp$ and $V_r^\perp$, respectively)
\begin{align}
\text{Left nullspace: } \text{span}(U_r^\perp) &= \text{span}(I_{d_Q} - U_rU_r^T), \\
\text{Right nullspace: } \text{span}(V_r^\perp) &= \text{span}(I_{d_M} - V_rV_r^T).
\end{align}
\end{subequations}
In the case of the truncated SVD the orthogonal complements $U_r^\perp$ and $V_r^\perp$ do not exactly coincide with the left and right nullspaces, but the map is locally uninformed in these modes since the singular values are small, in this case we refer to these as the uninformed subspaces.

A first idea we consider is to reduce the complexity of the Jacobian approximation error computation at a sample point using the truncated SVD. In this approach, the goal is to match the derivatives of the neural operator to the derivatives of $q(m)$ in the dominant subspaces defined by the left and right singular vectors of the truncated SVD at each parameter sample; this aims improve the parametric map approximation by enforcing smooth interpolation of the map between training data. Additionally, this approach can be used to achieve accurate approximations of the full Jacobian, if it is used in tandem with an efficient approach to penalize the singular values in the uninformed complementary subspaces of the truncated SVD. This penalization is difficult to achieve in general as the complementary subspace conditions are formally high-dimensional unlike the truncated Jacobian approximation error. We start with a proposition for the decomposition of the Jacobian error into low-dimensional conditions in informed subspaces and conditions for the complementary subspaces.

\begin{proposition}{Controlling the Jacobian approximation error with truncated SVD} \label{prop:decompose_h1_seminorm}. 
At every parameter sample the following error bound holds:
\begin{align}
    \|\nabla q - \nabla f_w\|^2_{F(\mathbb{R}^{d_Q \times d_M})} \leq  &\|\Sigma_r - U_r^T\nabla f_wV_r\|^2_{F(\mathbb{R}^{r \times r})} \nonumber + \sum_{i=r+1}^{\min\{d_Q,d_M\}}\sigma_i^2 \\
    +&\| (I_{d_Q} - U_rU_r^T)\nabla f_w (I_{d_M}- V_rV_r^T)\|^2_{F(\mathbb{R}^{d_Q \times d_M})} \nonumber \\
     +&\|(I_{d_Q} - U_rU_r^T)\nabla f_wV_r\|^2_{F(\mathbb{R}^{d_Q \times r})}\nonumber\\
     +&\|U_r^T\nabla f_w (I_{d_M}- V_rV_r^T)\|^2_{F(\mathbb{R}^{r \times d_M})}.
\end{align}

\end{proposition}




The proof is in Appendix \ref{section:proof_of_h1_seminorm}. This proposition demonstrates that the dominant Jacobian information represented by the truncated SVD can be efficiently encoded into neural operator regression for only $O(r^2 + rd_Q + rd_M)$, by directly penalizing the square of the truncated Jacobian approximation error $\|\Sigma_r - U_r^T\nabla f_wV_r\|_{F(\mathbb{R}^{r \times r})}$, but it comes with caveats. First if the Jacobian is over-truncated, the trailing sum of the squared singular values will lead to errors in approximating the true Jacobian approximation error. Additionally this bound shows that in order to control the Jacobian approximation error, one needs to somehow penalize the norm of the neural operator Jacobian in the uninformed subspaces that are orthogonal to the truncated SVD subspaces. Penalizing these terms formally scales with the product of the dimensions $d_Q\times d_M$, and is likely infeasible in many settings. As will be shown in numerical results, this truncated Jacobian error minimization still has positive benefits: it can substantially improve the function approximation, and can lead to accurate Jacobian computations in the dominant subspaces captured by the truncated SVD.

When $r$ is large, computing the Frobenius norm in $\mathbb{R}^{r\times r}$ may still be untenable. This is because automatic differentiation requires significant memory and operational overhead, while stochastic optimization requires that additional memory be allocated for additional arrays used in evaluating finite sum gradients and Hessians used in the training of the neural operator. We propose to further ease this computational burden by randomly subsampling left and right singular vectors at each epoch of training. In the following proposition we discuss the approximation of the Frobenius norm by row and column subsampling.

\begin{proposition}{Approximation of submatrix truncated SVD error norm in expectation}\label{prop:submatrix_svd_convergence}

Let $\kappa$ be a distribution of $k$ indices subsampled from $1,\dots r$ with uniform probability. Denote samples of these indices by $[\widehat{k}],[\widetilde{k}]$, and the subsampled matrices by $U_{[\widehat{k}]}\in\mathbb{R}^{d_Q \times k },V_{[\widetilde{k}]} \in \mathbb{R}^{d_M \times k}$. When $k$ columns of $U_r$ and $V_r$ are subsampled independently then

\begin{equation}
    \mathbb{E}_{[\widehat{k}],[\widetilde{k}] \sim \kappa} \left[\|\Sigma_{[\widehat{k}],[\widetilde{k}]}  - U_{[\widehat{k}]}^T\nabla f_w V_{[\widetilde{k}]}\|^2_{F(\mathbb{R}^{k \times k})}\right] = \frac{k^2}{r^2}\|\Sigma_r - U_r^T\nabla f_w V_r\|^2_{F(\mathbb{R}^{r \times r})}.
\end{equation}
Here $\Sigma_{[\widehat{k}],[\widetilde{k}]}  = U_{[\widehat{k}]}^TU_r\Sigma_r V_r^TV_{[\widetilde{k}]}$. When the same $k$ columns are subsampled for both $U_r$ and $V_r$, then
\begin{align}
    \mathbb{E}_{[\widehat{k}] \sim \kappa} \left[\|\Sigma_{[\widehat{k}],[\widehat{k}]}  - U_{[\widehat{k}]}^T\nabla f_w V_{[\widehat{k}]}\|^2_{F(\mathbb{R}^{k \times k})}\right] = \frac{k}{r}\|\text{diag}(\Sigma_r - U_r^T\nabla_m f_w V_r)\|^2_{F(\mathbb{R}^{r \times r})}\nonumber \\
     + \frac{k(k-1)}{r(r-1)}\|\text{offdiag}(\Sigma_r - U_r^T\nabla f_w V_r)\|^2_{F(\mathbb{R}^{r \times r})}.\label{eq:dependent_samples}
\end{align}
\end{proposition}

The proof is in Appendix \ref{section:proof_of_subsampled_svd}. This proposition establishes that using either independent or dependent samples for rows and columns, the subsampled approximation of the Frobenius norm of the truncated Jacobian approximation error serves as an asymptotically equivalent approximation of the truncated Jacobian approximation error, which can be scalably incorporated into training independent of how large $r$ is. In the case that the same samples are used for both the rows and columns, the diagonals of the matrix are over-emphasized. Since $\Sigma_r$ is by definition a diagonal matrix, this can be a beneficial property since it guarantees that there will always be $k$ nonzero slope conditions that are imposed at each sample. In the case that independent samples are used it is possible that only off diagonal elements of $\Sigma_r$ will show up in the loss function, which in our experience can make the optimization problem harder. Proposition \ref{prop:submatrix_svd_convergence} demonstrates the truncated Jacobian approximation error can be computed in a memory efficient way, even when $r$ is large. 

The truncated Jacobian approximation error is generic and can be included in neural operator training in a matrix-free way. The low-dimensional information can be extracted from generic differentiable neural operators $f_w(m)$ via use of automatic differentiation; a schematic for this is shown in Figure \ref{derivative_schematic}.

Numerical results in Section \ref{section:numerical_results} demonstrate that the inclusion of the truncated (i.e. dominant) Jacobian information, particularly with matrix subsampling, improves the parametric map approximation in the $L^2_\nu$ sense. Similar to Hermite interpolation methods, we expect that the inclusion of dominant Jacobian information at training data points leads to more accurate interpolation between training data points in the directions that the parametric map is changing the most. This makes the strategy desirable in contexts where one cares most about more accurate parametric map approximations. 

Unfortunately, only including the truncated Jacobian information in the neural operator training does not guarantee accurate Jacobian approximations, since it does not directly address the orthogonal complementary subspaces, which are naturally high-dimensional. Instead, in what follows we show that by using reduced basis neural operator architectures, one can exploit low dimensionality of the Jacobian information between the input and output bases. We will show that these neural operators naturally reduce the dimensionality of derivative learning, allowing one to efficiently learn both the dominant Jacobian information and the complimentary nullspace conditions, independent of the discretization dimensions $d_Q,d_M$

\subsection{Reduced Basis Derivative Learning} \label{section:arch_constraints}

Consider the reduced basis neural network
\begin{equation}\label{ridge_function}
    f_w^{\bar{r}}(m) = \Phi_{\bar{r}_Q }\phi_{\bar{r}}(\Psi_{\bar{r}_M}^Tm,w) + b,
\end{equation}
where $\phi_{\bar{r}}$ parametrizes a nonlinear operator between the two reduced subspaces with weights $w$, where $\Phi_{\bar{r}_Q} \in \mathbb{R}^{d_Q\times \bar{r}_Q}$ and $\Psi_{\bar{r}_M} \in \mathbb{R}^{d_M \times \bar{r}_M} $ are orthonormal reduced bases for the inputs and outputs, respectively, for the map $m \mapsto q$ and its Jacobian\footnote{We adopt the notation $\Phi_{\bar{r}_Q} $ and $\Psi_{\bar{r}_M}$  in order to distinguish these global bases with global basis ranks $\bar{r}_Q$ and $\bar{r}_M$ from the local bases defined by the truncated SVD with rank $r$ discussed in earlier setions.}. The vector $b$ represents an affine shift in the output space, and can be learned in conjunction with the weights $w$, or can be supplied as e.g., the mean of the output training data. Examples of these reduced basis architectures include PCANet \cite{BhattacharyaHosseiniKovachkiEtAl2021} where the bases for the inputs and outputs are computed using principal component analysis (PCA),derivative-informed projected architectures \cite{OLearyRoseberryDuChaudhuriEtAl2021,OLeary-RoseberryVillaChenEtAl22} where derivative-informed bases are used, and POD-DeepONet \cite{lu2022comprehensive} which applies ideas related to PCANet to the DeepONet architecture. A property of these functions is that their Jacobian with respect to $m$ is structurally zero in the orthogonal complements of the input and output reduced bases, as we show next.

\begin{proposition}{Orthogonality Conditions for Reduced Basis Neural Operator Jacobian.}\label{prop:jacobian_ridge_orthogonality}
Without loss of generality take $\bar{r}=\bar{r}_Q = \bar{r}_M$.
\begin{enumerate}
\item If $X \perp \Psi_{\bar{r}}$, then $\nabla f^{\bar{r}}_w X = 0$.

\item If $Y \perp \Phi_{\bar{r}}$, then $Y^T \nabla f^{\bar{r}}_w = 0$.
\end{enumerate}
\end{proposition}

\begin{proof}
By the chain rule we have
\begin{equation}
    \nabla f^{\bar{r}}_w = \Phi_{\bar{r}} \nabla_z\phi_{\bar{r}}(z,w)|_{z = \Psi_{\bar{r}}^Tm}\Psi_{\bar{r}}^T.
\end{equation}
Since $\Psi_{\bar{r}}^TX = 0$ we get that $\nabla f^{\bar{r}}_wX = 0$. Since $Y^T\Phi_{\bar{r}} = 0$, we get that $Y^T \nabla f^{\bar{r}}_w = 0$.
\end{proof}

This immediately leads to the following result about the dimension independence of the Jacobian training.

\begin{theorem}{Dimension Independence of Reduced Basis Derivative Learning.}\label{theorem:dim_independent_dino}
Let $m_r = \Psi_{\bar{r}_M}^Tm$. Then we have the equivalence of the following optimization problems:
\begin{equation}
  \min_w \|\Phi_{\bar{r}_Q}^T\nabla q \Psi_{\bar{r}_M}- \nabla_{m_r}\phi_{\bar{r}}\|^2_{F(\mathbb{R}^{\bar{r}_Q\times \bar{r}_M})} \Longleftrightarrow \min_w\|\nabla q  - \nabla f_w^{\bar{r}} \|^2_{F(\mathbb{R}^{d_Q\times d_M})}
\end{equation}
\end{theorem}

\begin{proof}
By applying the orthogonality identity \eqref{orth_decomp_identity} with respect to both $\Phi_{\bar{r}_Q}$ and $\Psi_{\bar{r}_Q}$to $\|\nabla q  - \nabla f_w^{\bar{r}} \|^2$, we can decompose it into $\|\Phi_{{\bar{r}_Q}}^T\nabla q \Psi_{{\bar{r}_M}}- \Phi_{\bar{r}_Q}^T\nabla f_w^{\bar{r}}\Psi_{\bar{r}_M}\|^2_{F(\mathbb{R}^{\bar{r}_Q\times \bar{r}_M)}}$ and terms corresponding to orthogonal complements of $\Phi_{{\bar{r}_Q}}$ and $\Psi_{\bar{r}_M}$. By proposition \ref{prop:jacobian_ridge_orthogonality} all terms involving $\nabla f^{\bar{r}}$ are annihilated, and thus the $w$ gradient of these terms is zero, because these terms are merely affine shifts of the objective function. The result follows from the chain rule and the orthonormality of the two reduced bases:
\begin{align}
  \Phi_{\bar{r}_Q}^T\nabla f^{\bar{r}}\Psi_{\bar{r}_M} = \Phi_{\bar{r}_Q}^T\Phi_{{\bar{r}_Q}}\nabla_{m_r}\phi_{\bar{r}}\Psi_{{\bar{r}_M}}^T\Psi_{{\bar{r}_M}} = \nabla_{m_r}\phi_{\bar{r}}.
\end{align}
\end{proof}

The consequence of this theorem are that the costs of Jacobian training data generation require only $\min\{\bar{r}_Q,\bar{r}_M\}$ linearized forward or adjoint PDE solves per training point, and additionally only the neural network mapping between the reduced subspaces, $\phi_{\bar{r}}$, needs to be built for the training. Thus the online memory and computational costs can be made independent of $d_M$ and $d_Q$ by projecting all training data into the reduced bases. See Figure \ref{rb_dino_schematic} for a schematic of the reduced basis neural operator training.

Proposition \ref{prop:jacobian_ridge_orthogonality} establishes that the reduced basis neural operator has a structurally zero Jacobian in the orthogonal complements of both the reduced input and output bases. This makes clear the significance of the choice of reduced bases; one wants to choose subspaces that resolve the left and right singular vectors of the Jacobian in expectation over parameter space, while tending to annihilate the Jacobian left and right nullspaces. When the forward map and its derivatives can be well-approximated by a reduced basis neural operator, Theorem \ref{theorem:dim_independent_dino} establishes the dimension-independence of reduced basis neural operator derivative training.

To construct a basis for the parameter and QoI spaces, we proceed as follows. The right singular vectors of the Jacobian (weighted by the square of the singular values) are captured by the derivative-informed input basis (i.e., active subspace \cite{ZahmConstantinePrieurEtAl2020}), which is the dominant eigenvectors of
\begin{equation} \label{eq:active_subspace}
\mathbb{E}_\nu[\nabla q^T \nabla q] \in \mathbb{R}^{d_M \times d_M}.
\end{equation}
The derivative-informed input basis is a powerful derivative based dimension reduction technique that resolves the dominant sensitivity information of a parametric map in expectation. It has been useful in dimension reduction techniques and surrogate modeling \cite{ConstantineDowWang2014,OLeary-RoseberryVillaChenEtAl22,OLearyRoseberryDuChaudhuriEtAl2021,ZahmConstantinePrieurEtAl2020}. Similarly, a basis for the QoI space can be constructed via the dominant information contained in the left singular vectors of the Jacobian (again weighted by the square of the singular values) by computing the dominant eigenvectors of:
\begin{equation}\label{eq:derivative_output_basis}
    \mathbb{E}_\nu[\nabla q \nabla q^T] \in \mathbb{R}^{d_Q \times d_Q}.
\end{equation}
We refer to this basis as the derivative-informed output basis. By the use of Poincar\'{e} inequalities, errors in approximating $H^1_\nu$ maps by restricting to the first $r$ eigenvectors of \eqref{eq:active_subspace} can be bounded by the sum of the trailing eigenvalues corresponding to the complementary eigenvectors \cite{ZahmConstantinePrieurEtAl2020}. Recent work has provided error analysis for dimension reduced algorithms used to solve Bayesian inverse problems using bases similar to the derivative-informed output basis \cite{baptista2022gradient}. 

Reduced basis neural networks \eqref{ridge_function} constructed using derivative-informed input and output bases \eqref{eq:active_subspace} and \eqref{eq:derivative_output_basis} are referred to as derivative-informed projected neural networks (DIPNets) \cite{OLearyRoseberryDuChaudhuriEtAl2021,OLeary-RoseberryVillaChenEtAl22}. Given limited training data, DIPNets can achieve high accuracy in approximating parametric functions by focusing on the most sensitive directions in both parameter and data spaces. These architectures are well suited to derivative approximation since they are designed to represent global Jacobian information. The combination of DIPNets with derivative training in the combined reduced bases creates an accurate and scalable approach to approximating parametric maps and derivatives, as demonstrated by numerical results in the next section.

\section{Numerical Experiments} \label{section:numerical_results}

In this section we investigate the effects of including parametric derivative information in neural operator regression for parametric maps with PDE constraints. We compare three different derivative-informed optimization formulations (as proposed in Section \ref{section:dino}), with standard neural operator $L^2_\nu$ training; these formulations are summarized in \eqref{eq:opt_formulations} and Table \ref{table:method_costs}. We consider two different neural network architectures all for differing sizes of the training data set. We assess the accuracy of the trained neural operators in terms of generalization accuracy for the function, and additionally consider various metrics of derivative accuracy, both for the Jacobian itself, and additionally for other derived quantities (gradients and Gauss--Newton Hessians that show up in related inverse problems). In general, our numerical results demonstrate that including derivative information in neural operator training improves the $L^2_\nu$ generalization accuracy of the neural operator, which makes a case for our proposed methods independent of the need for accurate parametric derivatives. Additionally we demonstrate that neural operators trained without derivative information can result in inaccurate computations of high-dimensional derivative quantities (e.g., gradients and Gauss--Newton Hessians with respect to high-dimensional fields), and that the inclusion of derivative information in the training is essential to improve the accuracy of these quantities. The use of derivative-informed reduced basis networks led to both accurate function and derivative approximations, in particular when the amount of training data were limited. We additionally observe that as more training data are available (both function and parametric derivative), over-parametrized networks can yield favorable approximation of both functions and derivatives. The code used to generate these numerical results, and all related details can be found in the \texttt{dino} repository \url{https://github.com/tomoleary/dino} \cite{dino}. Specific instructions about how to generate the training data, train the neural operators, and post-process the trained neural operators, are contained in the \texttt{applications/} directory.

\subsection{Definition of Parametric PDE Maps}

We consider PDE problems where the mapping $m \mapsto q$ represents a mapping from a (high-dimensional) coefficient field parametrizing the PDE to the observables of the PDE state variable at points inside the domain $\Omega$ of the PDE. In all cases we consider unit square domains for two dimensional problems. We use a centered Gaussian distribution $\nu = \mathcal{N}(0,\mathcal{C})$ for the random field with a trace-class Mat\'{e}rn covariance \cite{LindgrenRueLindstrom2011} $\mathcal{C} = \mathcal{A}^{-2}$, where
\begin{align}\label{eq:matern_covariance}
    \mathcal{A} = \begin{cases} (\delta I - \gamma \Delta) \text{ in } \Omega \\
              -\gamma \nabla \cdot n_{\partial\Omega} \text{ on } \partial \Omega \end{cases} 
\end{align}
where $n_{\partial\Omega}$ here is the outward unit normal vector to the domain boundary $\partial \Omega$. The correlation structure is parametrized by $\delta,\gamma> 0$ with the correlation length dictated by the ratio $\sqrt{\gamma/\delta}$. For fixed correlation length, the marginal variance is decreased by making $\gamma$ and $\delta$ larger. We use a uniform mesh of size $64 \times 64 $, and linear finite elements with piecewise linear basis functions for the parameter discretization, leading to the input dimension, $d_M = 4,225$ for all problems. We use a pointwise observation operator on the state at locations $\{x_i \in \Omega\}_{i=1}^{n_\text{obs}}$. This parameter-to-observable map is a critical component of Bayesian inverse problems and Bayesian optimal experimental design problems. In the first two problems, we only observe the state in only half of the domain, in order to imitate partial spatial observations which are typical in Bayesian inverse and optimal experimental design problems. We consider three PDE problems as follows.

\subsubsection{Reaction-Diffusion Problem}

For the first test case, we consider a nonlinear reaction-diffusion problem in $\Omega = (0,1)^2$ with a lognormal diffusion coefficient field $m \sim \mathcal{N}(0,\mathcal{C})$. In this PDE formulation the right hand side source term, $s$, is a sum of $25$ smoothed point sources located on a Cartesian grid centered in $\Omega$, and additionally there is a cubic nonlinearity in the reaction term. We take $\delta = 1.0,\gamma =0.1$. In this problem, $n_\text{obs} = 50$ and the state variable $u$ is scalar valued so $d_Q = 50$, as shown in Figure \ref{fig:rdiff_state}.

\begin{align}
-\nabla \cdot(e^{m}\nabla u) + u^3 &= s \text{ in } \Omega \nonumber\\
u &= 1 \text{ on } \Gamma_\text{top} \nonumber \\
e^m\nabla u \cdot n &= 0 \text{ on } \Gamma_\text{sides} \nonumber\\
u &= 0 \text{ on } \Gamma_\text{bottom} 
\end{align}

\begin{figure}[H]
\center
\includegraphics[width = 0.6\textwidth]{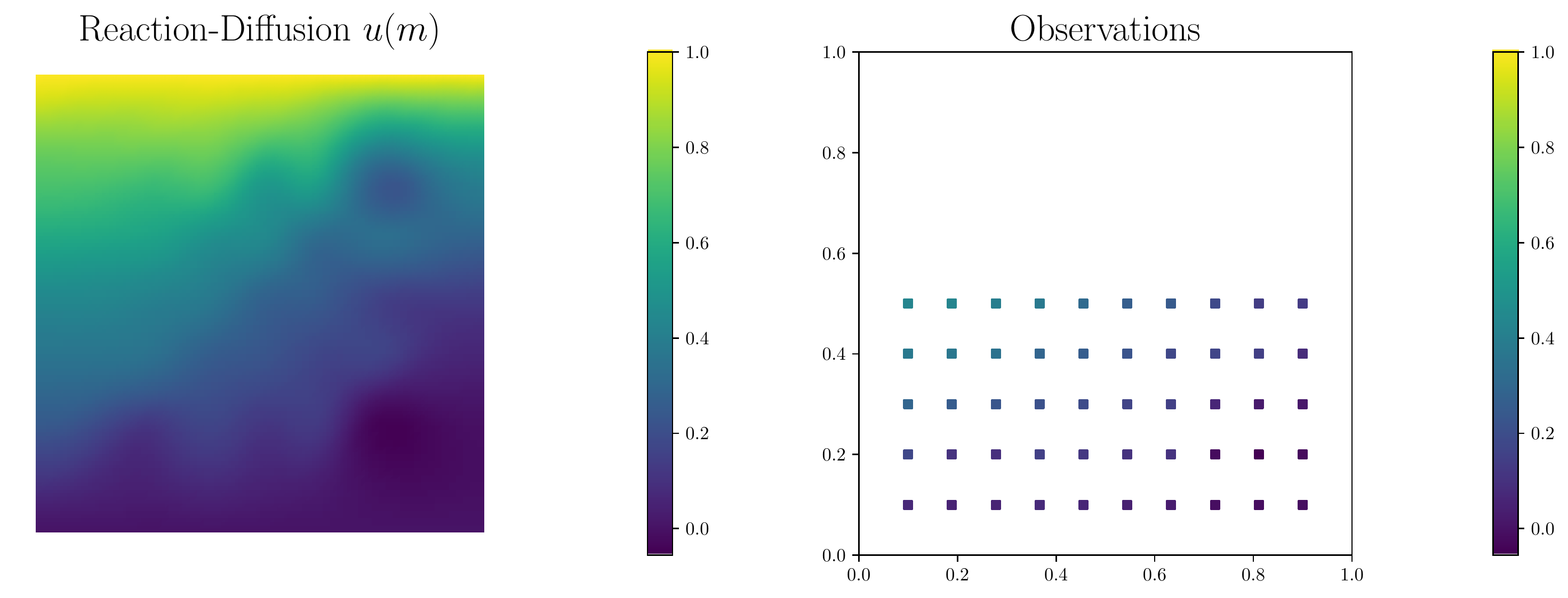}
\caption{An instance of reaction-diffusion state and observables}
\label{fig:rdiff_state}
\end{figure}

\subsubsection{Convection-Reaction-Diffusion Problem}

The second test case is a convection-reaction-diffusion (CRD) problem where the parameter shows up in a nonlinear reaction term \cite{OLeary-RoseberryVillaChenEtAl22,WuOLearyRoseberryChenEtAl23}. The parameters for the distribution are $\delta = 1.0,\gamma = 0.1$. The quantity of interest is again a grid of observables of the PDE solution $u(x_i)$ in the lower half of the domain (the same as in Figure \ref{fig:rdiff_state}). The right hand side of the PDE $s$,  is given by a Gaussian bump centered at $(0.7,0.7)$. The velocity field $v$ is given by a solution to a steady-state Navier-Stokes equation with boundary conditions driving the flow, for more information see the Appendices of \cite{OLeary-RoseberryVillaChenEtAl22}. See Figure \ref{fig:crd_state} for an instance of a parameter and the corresponding state.

\begin{align}
  - \nabla \cdot (k \nabla u) + v \cdot \nabla u  + e^m u^3 &= s \quad \text{in } \Omega \nonumber \\ 
  u &= 0 \text{ on } \partial \Omega \nonumber \\
\end{align}

\begin{figure}[H]
\center
\includegraphics[width = 0.6\textwidth]{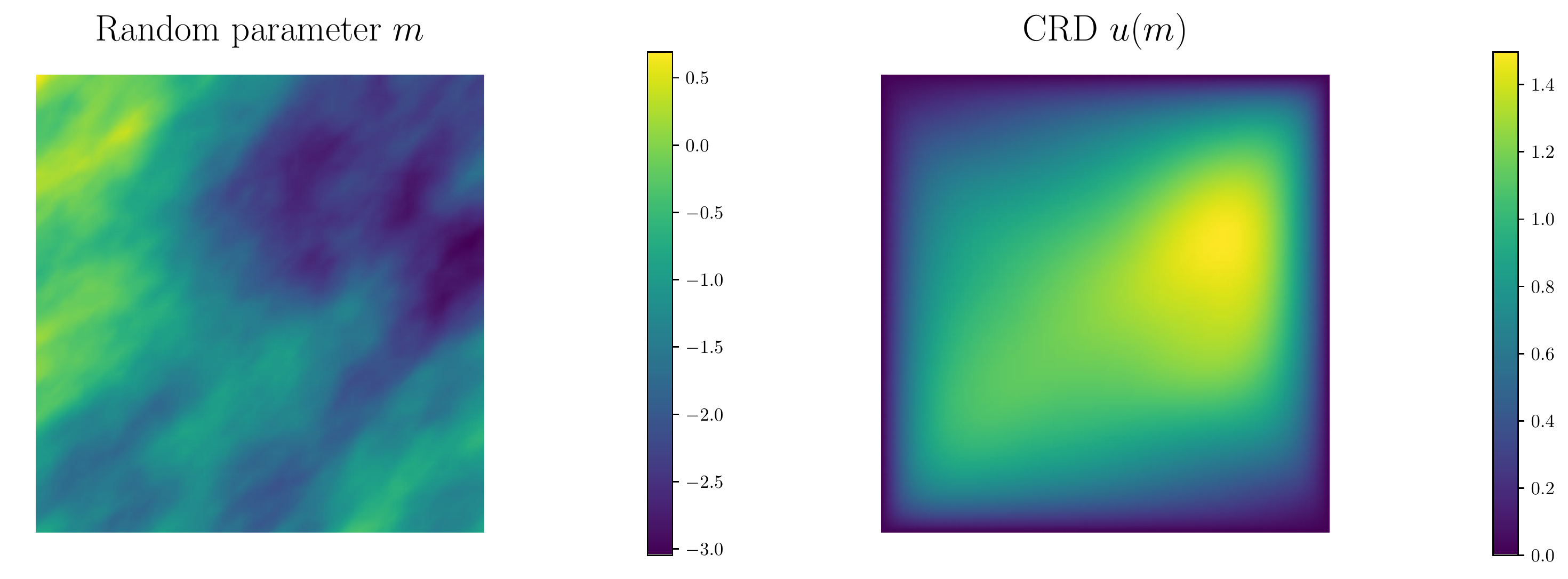}
\caption{An instance of CRD random parameter and corresponding state.}
\label{fig:crd_state}
\end{figure}

\subsubsection{Hyperelasticity Problem}

The third test case is the deformation of a hyperelastic material with spatially varying Young's modulus that is adapted from \cite{CaoOLearyRoseberryJhaEtAl2022} and is relevant to engineering and medical applications \cite{affagard2015identification, goenezen2011solution, gokhale2008solution,   mei2018comparative}. The forward problem is a nonlinear elastic model for the deformation of a neo-Hookean material \cite{GonzalezStuart08} in $\Omega $ under a prescribed physical load. We describe the PDE model here in brief. Let $X$ represent a material point, which is mapped to a spatial point via the displacement $u(x)$: $x = I + u(X)$. The deformation of the hyperelastic material coincides with a change in the internal stress state of the material. This is captured by the strain energy density $W = W(X,C)$, a function of both the material coordinate $X$ as well as the right Cauchy-Green stress tensor $C = F^TF$, which is defined in terms of the deformation gradient $F=I+\nabla u$. For neo-Hookean material \cite{GonzalezStuart08} the strain  energy density function is given by 
\begin{subequations}
\begin{equation}
  W(X,C) = \frac{\mu(X)}{2}(\text{tr}(C) - 3) + \frac{\lambda(X)}{2}\ln(J)^2 - \mu(X)\ln(J),
\end{equation}
where $\text{tr}$ is the trace, and $J = \text{det}(F)$. The Lam\'{e} parameters $\lambda,\mu$ are related to the spatially varying Young's modulus $E$ and the Poisson's ratio $\nu_P$ by
\begin{equation}
  \lambda(X) = \frac{E(X)\nu_P}{(1+\nu_P)(1-2\nu_P)}, \qquad \mu = \frac{E(X)}{2(1+\nu_P)}.
\end{equation}
We assume $\nu_P = 0.4$ is constant, and that the Young's modulus follows a log-normal distribution as follows:
\begin{equation}
  E = \exp(m) + E_0, \qquad m \sim \mathcal{N}(m_0,\mathcal{C}),
\end{equation}
where the paramter $E_0= 1.0$, and $m_0 = 0.37$ is a constant over the domain $\Omega$. The choices of $\gamma$ and $\delta$ in \eqref{eq:matern_covariance} are $\gamma = \frac{4}{3}$ and $\delta = 0.12$. Assuming external forces are applied slowly, and temporal effects in the deformation of the material are negligle, the steady state balance of linear momentum results in the following PDE system.

\begin{align}
  \nabla \cdot (FS) &= 0 \text{ in } \Omega \\
  u &= 0 \text{ on } \Gamma_\text{left} \\
  FS \cdot n &= 0 \text{ on } \Gamma_\text{top} \cup\Gamma_\text{bot} \\
  FS\cdot n &= t \text{ on } \Gamma_\text{right}
\end{align}

\begin{figure}[H]
\center
\includegraphics[width = 0.288\textwidth]{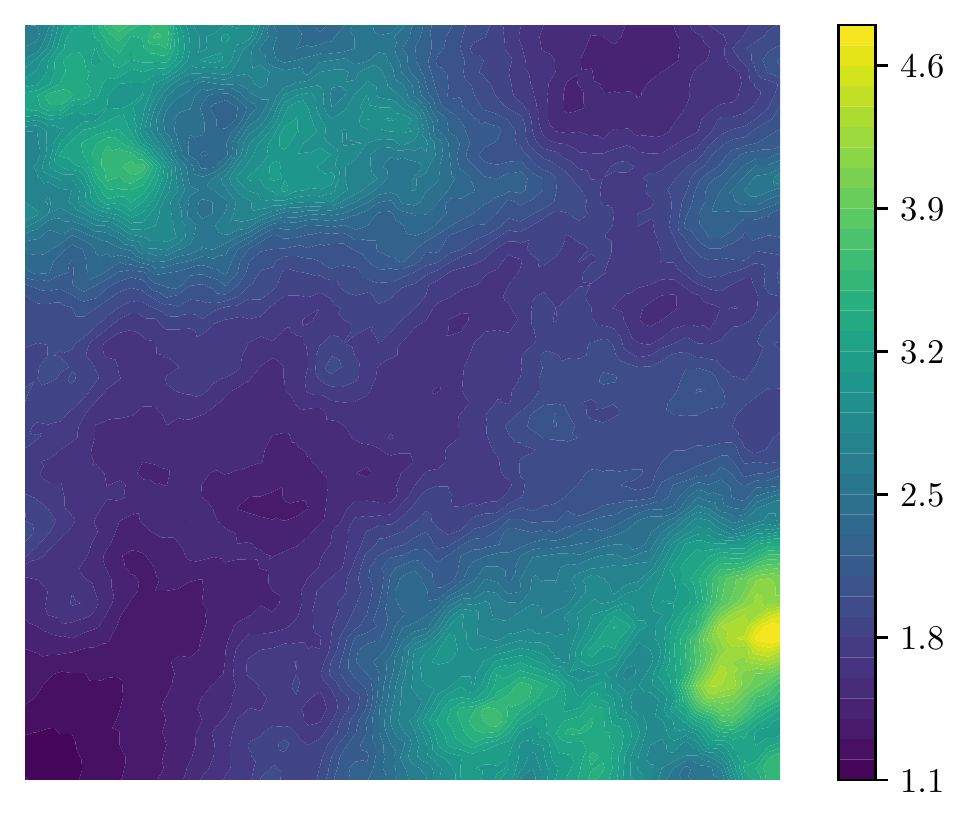}%
\includegraphics[width = 0.312\textwidth]{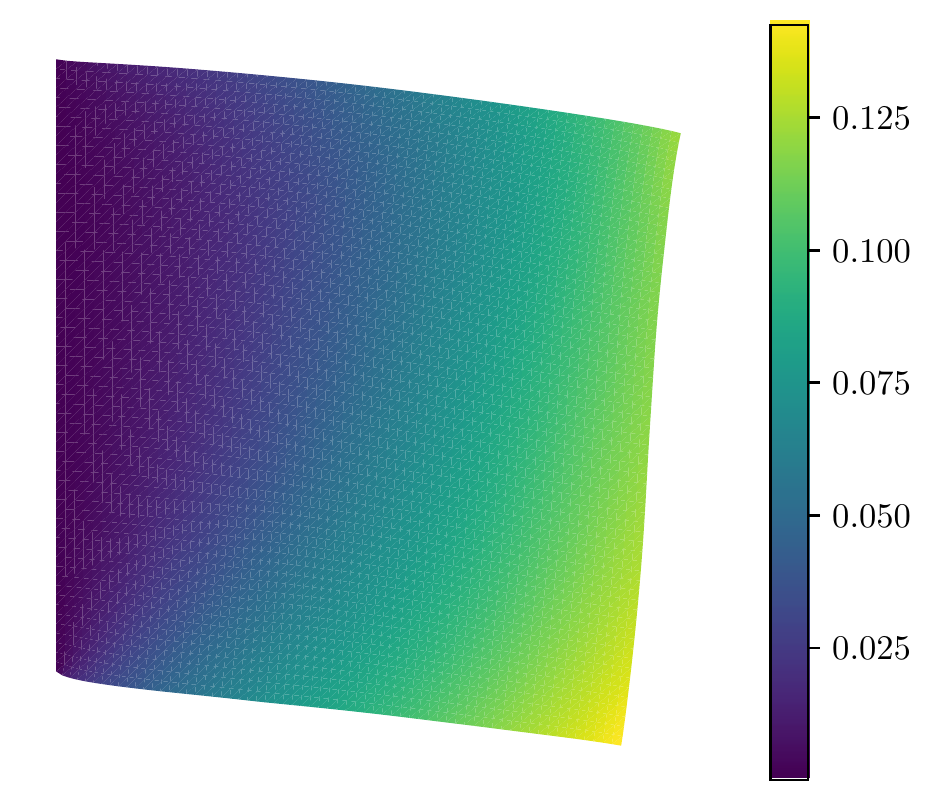}

\caption{An instance of a random elastic modulus (left) and the corresponding displacement state (right).}
\label{fig:hyperelasticity_state}
\end{figure}

$S$ is the second Piola-Kirchhoff stress tensor which represents the stress in the reference configuration, and is given by $S(X,C) = 2 \frac{\partial W(X,C)}{\partial C}$, and the prescribed traction load $t$ is a combination of a Gaussian compressive component and a linear shear state: $t(X) = 0.06\text{exp}\left(-0.25|X_2 - 0.5|^2\right)e_1 + 0.03(1+0.1X_2)e_2$. As with the reaction-diffusion and CRD problems, there are $50$ observations on a grid as in Figure \ref{fig:rdiff_state}, the only differences in this case are that the vertical spacing between observations are doubled so as to span more of the material domain horizontally. Additionally since the state variable is a vector in $\mathbb{R}^{2}$, the observation dimension is doubled to $d_Q = 100$.

\end{subequations}

\subsection{Overview of Architectures and Training}\label{section:training_and_accuracies}

We consider two different networks architectures for each problem, a generic encoder-decoder network (``Generic'') and a derivative-informed projected network (DIPNet) \cite{OLeary-RoseberryVillaChenEtAl22}. The first network captures effects due to over-parametrization (e.g., poor generalization given limited training data), while the latter illustrates the benefits of reduced-basis neural operators as discussed in Section \ref{section:arch_constraints}. For the DIPNet architectures we mainly consider one network that uses a $2\times d_Q$ dimensional reduced basis for the inputs, and a full $d_Q$ dimensional basis for the outputs. For the reaction-diffusion and CRD problems we refer to these as ``DIPNet 100-50'', and for the hyperelasticity ``DIPNet 200-100''. In assessing the accuracies of derivative quantities (gradients and Gauss--Newton Hessians), we additionally consider the effects of reducing the input basis dimension to $d_Q$. The generic encoder-decoder network in each problem uses the latent dimension of $d_Q$. All architectures have six hidden layers, and use softplus activation functions. For the reaction-diffusion and CRD problems, the trainable dimensions are $d_W = 20,400$ for DIPNet 50-50, $d_W = 30,450$ for DIPNet 100-50, and $d_W = 226,600$ for Generic. For the hyperelasticity problem, the trainable dimensions are $d_Q = 90,900$ for the DIPNet 100-100, $d_W = 321,600$ for the DIPNet 200-100, and $d_W = 1,086,400$ for Generic.

\begin{subequations}\label{eq:opt_formulations}
\begin{align}
    &L^2_\nu: \qquad\qquad \min_w \mathbb{E}_\nu\left[\|q - f_w\|^2_{2}\right]\\
    &\text{Full } H^1_\nu: \quad \begin{cases}\begin{aligned} &\min_w \mathbb{E}_\nu\left[\|q - f_w\|^2_{2} + \|\nabla q - \nabla f_w\|^2_{F(\mathbb{R}^{d_Q\times d_M})}\right] \text{, \qquad \qquad \enskip Generic}\\ &\min_w \mathbb{E}_\nu\left[\|q - f_w\|^2_{2} + \|\Phi_{\bar{r}_Q}^T(\nabla q - \nabla f_w)\Psi_{\bar{r}_M}\|^2_{F(\mathbb{R}^{\bar{r}_Q\times \bar{r}_M})}\right] \text{, DIPNet (see Theorem \ref{theorem:dim_independent_dino})}\end{aligned}\end{cases}
\end{align}
\begin{align}
    &\text{Truncated } H^1_\nu: \quad \qquad \min_w \mathbb{E}_\nu\left[\|q - f_w\|^2_{2} + \|U_r^T(\nabla q - \nabla f_w)V_r\|^2_{F(\mathbb{R}^{r\times r})}\right]\text{ (see Proposition \ref{prop:decompose_h1_seminorm})}\\
    &\text{Truncated } H^1_\nu \text{ MS}: \quad \min_w \mathbb{E}_\nu\left[\|q - f_w\|^2_{2} + \mathbb{E}_{[\widehat{k}] \sim \kappa}\left[\|U_{[\widehat{k}]}^T(\nabla q - \nabla f_w)V_{[\widehat{k}]}\|^2_{F(\mathbb{R}^{k\times k})}\right]\right] \text{ (see Proposition \ref{prop:submatrix_svd_convergence})}
\end{align}
\end{subequations}

For each network we consider four different formulations of the loss function (see equations \eqref{eq:opt_formulations}). The first is the typical $L^2_\nu$ training where no derivative information is included. The second is full $H^1_\nu$ parametric regression where the entire $H^1_\nu$ semi-norm loss term is trained with equal weighting to the $L^2_\nu$ loss (note that when training a reduced basis (RB) network (e.g., DIPNet) the online costs of evaluating the $H^1_\nu$ semi-norm are reduced to $O(\bar{r}_Q\bar{r}_M)$ due to Theorem \ref{theorem:dim_independent_dino}). The third formulation is the truncated $H^1_\nu$ parametric regression problem where the entire $H^1_\nu$ Jacobian semi-norm is replaced by a truncated $H^1_\nu$ semi-norm (the truncated Jacobian approximation error); this formulation is derived from Proposition \ref{prop:decompose_h1_seminorm}, but does not include the terms involving the Jacobian nullspaces. The fourth is the same but using matrix subsampling (MS) of the Jacobian with dependent row and column samples, as in Proposition \ref{prop:submatrix_svd_convergence}; this formulation also does not include the nullspaces.

A summary of the various computational costs per sample point of the proposed methods shown in Table \ref{table:method_costs}.

\begin{table}[h!]
\center
\begin{tabular}{|l|c|c|}
\hline
                                     & \multicolumn{1}{l|}{Offline (\# PDE solves)}        & \multicolumn{1}{l|}{Training costs per sample} \\ \hline
Full $H^1_\nu$                       & $\min\{d_M,d_Q\}$                   & $O(d_Qd_M)$                 \\ \hline
Truncated $H^1_\nu$                    & $O(r)$                              & $O(r^2 +rd_Q+rd_M)$                    \\ \hline
Truncated $H^1_\nu$ matrix subsampling & $O(r)$                              & $O(k^2+kd_Q+kd_M)$                    \\ \hline
 Full $H^1_\nu$ with reduced bases             & $\min\{\bar{r}_M,\bar{r}_Q\}$ & $O(\bar{r}_M\bar{r}_Q)$                 \\ \hline
\end{tabular}
\caption{Summary of the dominant offline and online computational costs for the various Jacobian learning methodologies proposed herein (per sample point). The offline costs are measured in terms of the number of additional PDE solves required to compute Jacobian matrices for training. The online costs are measured by the memory and computational costs of adding the Jacobian information to the neural operator training. }
\label{table:method_costs}
\end{table}

We train for $100$ epochs using an Adam optimizer \cite{KingmaBa2014}, with TensorFlow \cite{AbadiAgarwalBarhamEtAl2016} default hyperparameters (constant learning rate $\alpha = 10^{-3}$ and batch size of $32$) for simplicity. For each problem we use $1,024$ data for generalization tests, and study the effects of varying training data size.

\subsection{Improvement of the Function Accuracy via Derivative Training}

We begin by assessing $L^2_\nu$ accuracy, which is defined below
\begin{equation}
  L^2_\nu \text{ accuracy:} \quad  \left(1 - \sqrt{\mathbb{E}_\nu\left[\frac{\|q  - f_w\|^2_{\ell^2(\mathbb{R}^{d_Q})}}{\|q\|^2_{\ell^2(\mathbb{R}{^{d_Q}})}}\right]}\right).
\end{equation}

Figure \ref{l2_comps} shows that derivative information improved the function approximation in every case except the generic encoder-decoder with full $H^1_\nu$ training. A possible reason for this is that the generic encoder-decoder with full $H^1_\nu$ is the only combination of network architecture and training loss that is not discretization dimension-independent, so more training data may be needed for this method. Interestingly, for all architectures the matrix-subsampled truncated $H^1_\nu$ improved function accuracy the most, particularly given limited data. Possibly, this additional stochastic approximation in the derivative loss helps with the neural network training in a manner analogous to to why stochastic gradient descent is preferred over gradient descent. Perhaps due to the increased variance of the parameter distribution, the reaction-diffusion problem was generally much harder to learn than the CRD and hyperelasticity problems, to be discussed below

\begin{figure}
\begin{subfigure}{0.5\textwidth}
\includegraphics[width = \textwidth]{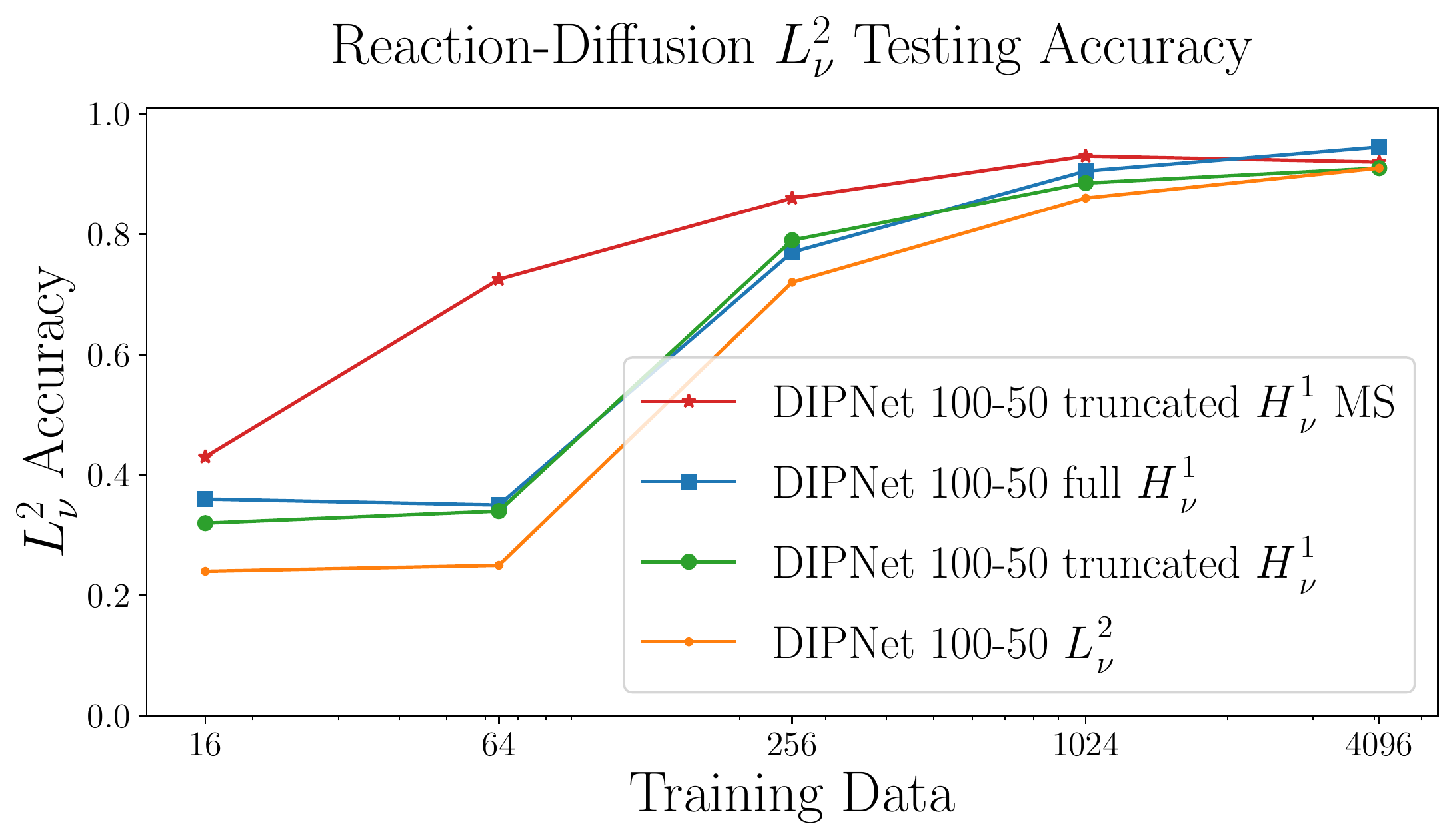}
\end{subfigure}%
\begin{subfigure}{0.5\textwidth}
\includegraphics[width = \textwidth]{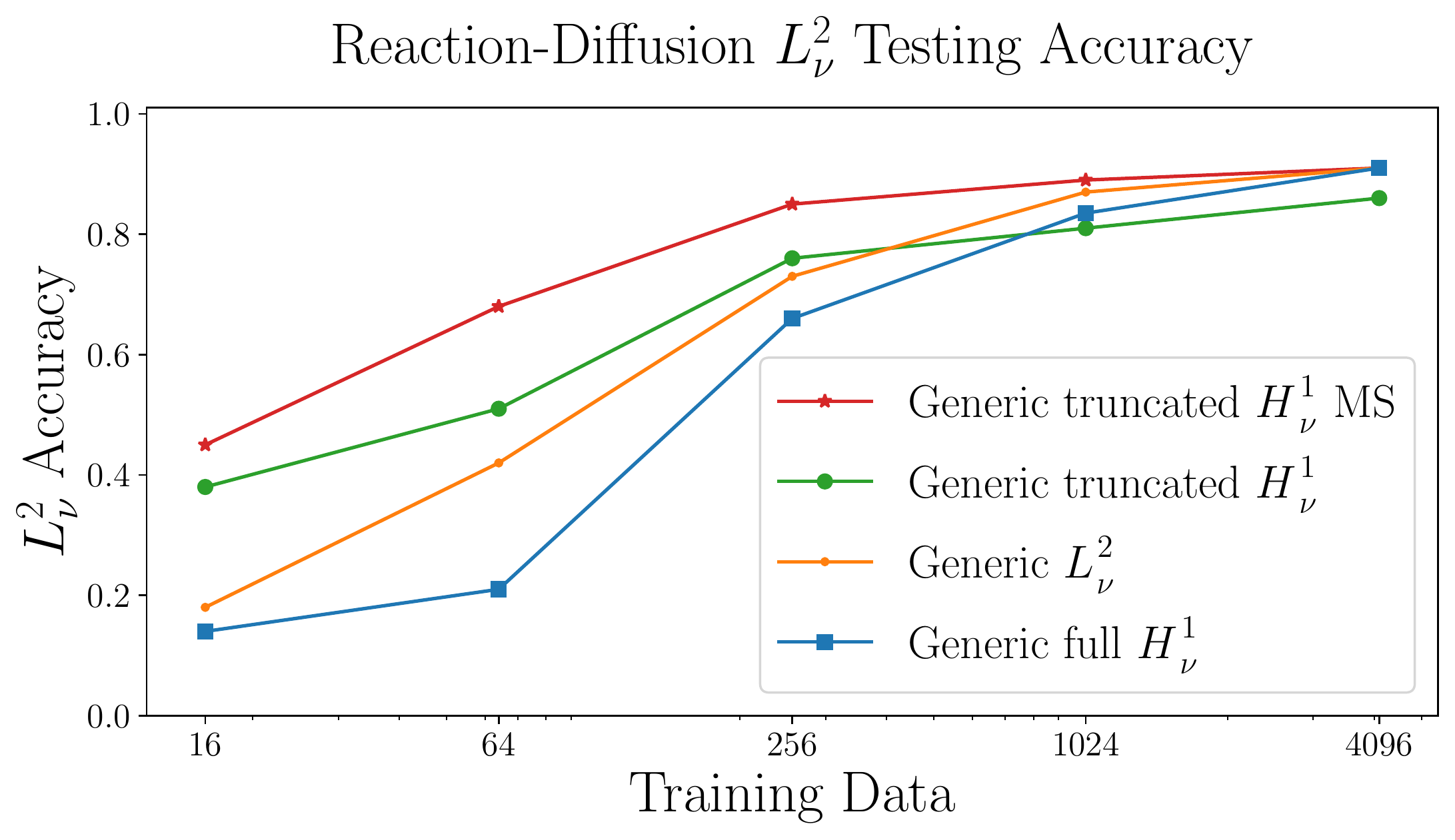}
\end{subfigure}
\begin{subfigure}{0.5\textwidth}
\includegraphics[width = \textwidth]{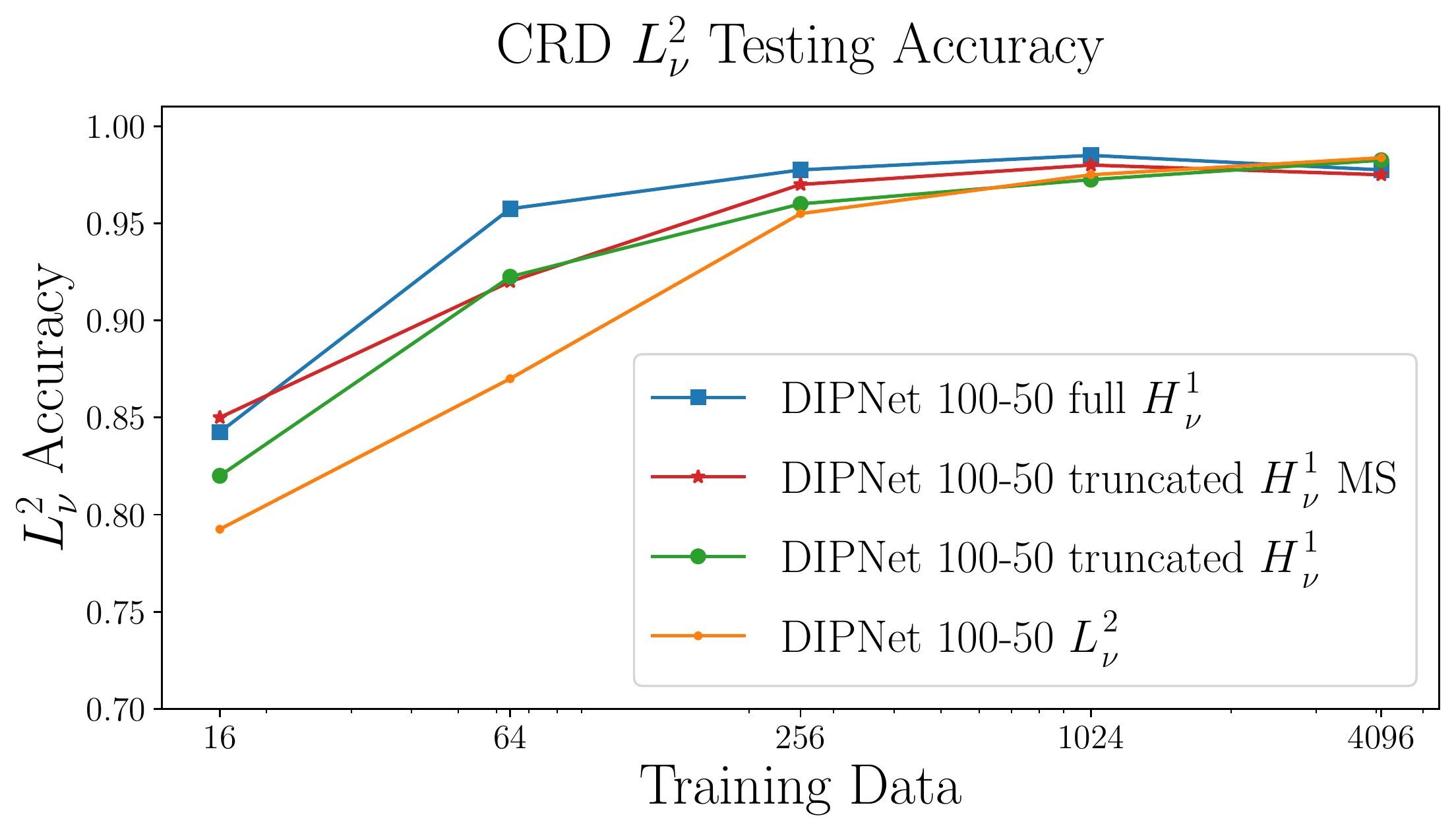}
\end{subfigure}%
\begin{subfigure}{0.5\textwidth}
\includegraphics[width = \textwidth]{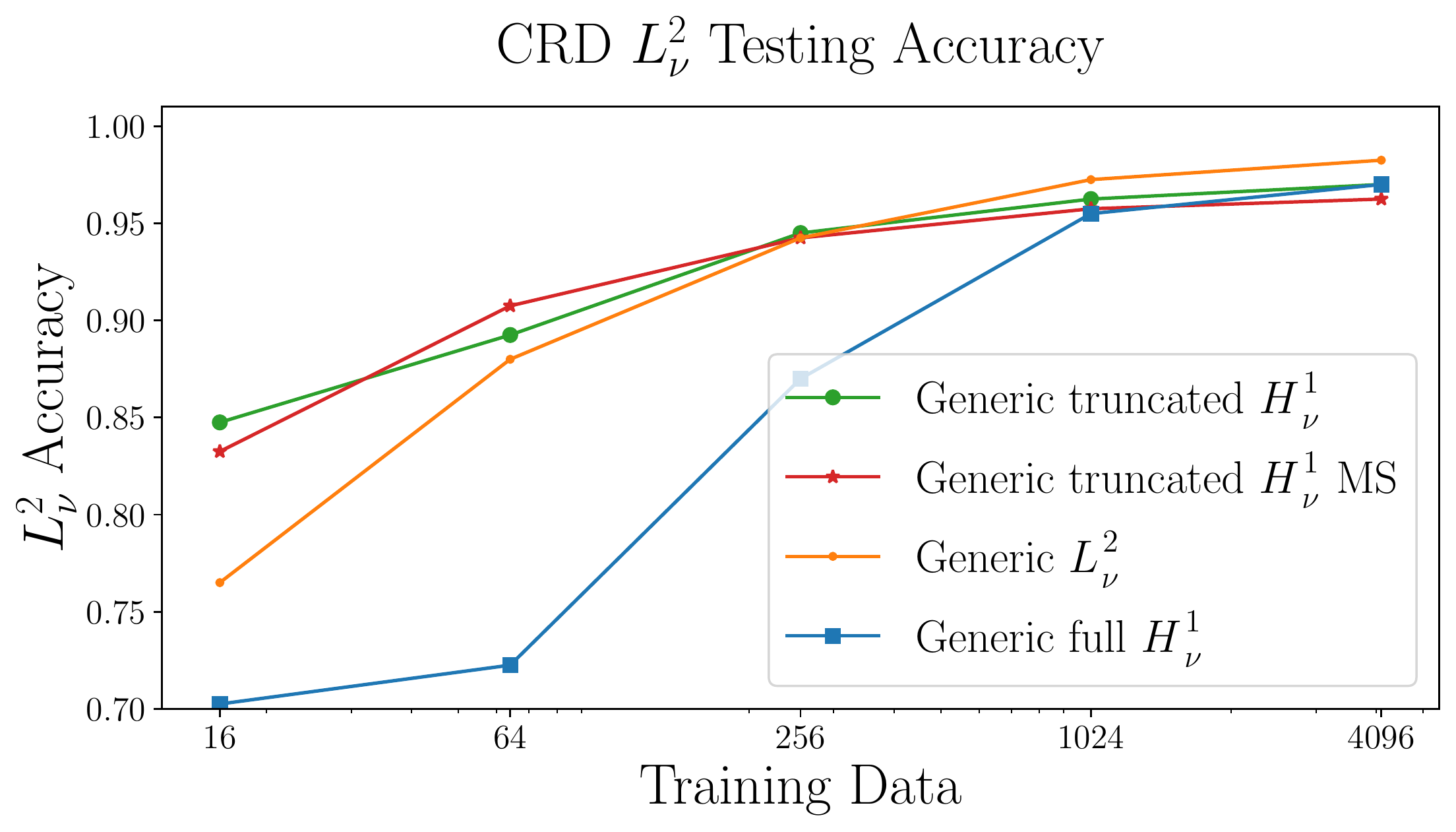}
\end{subfigure}
\begin{subfigure}{0.5\textwidth}
\includegraphics[width = \textwidth]{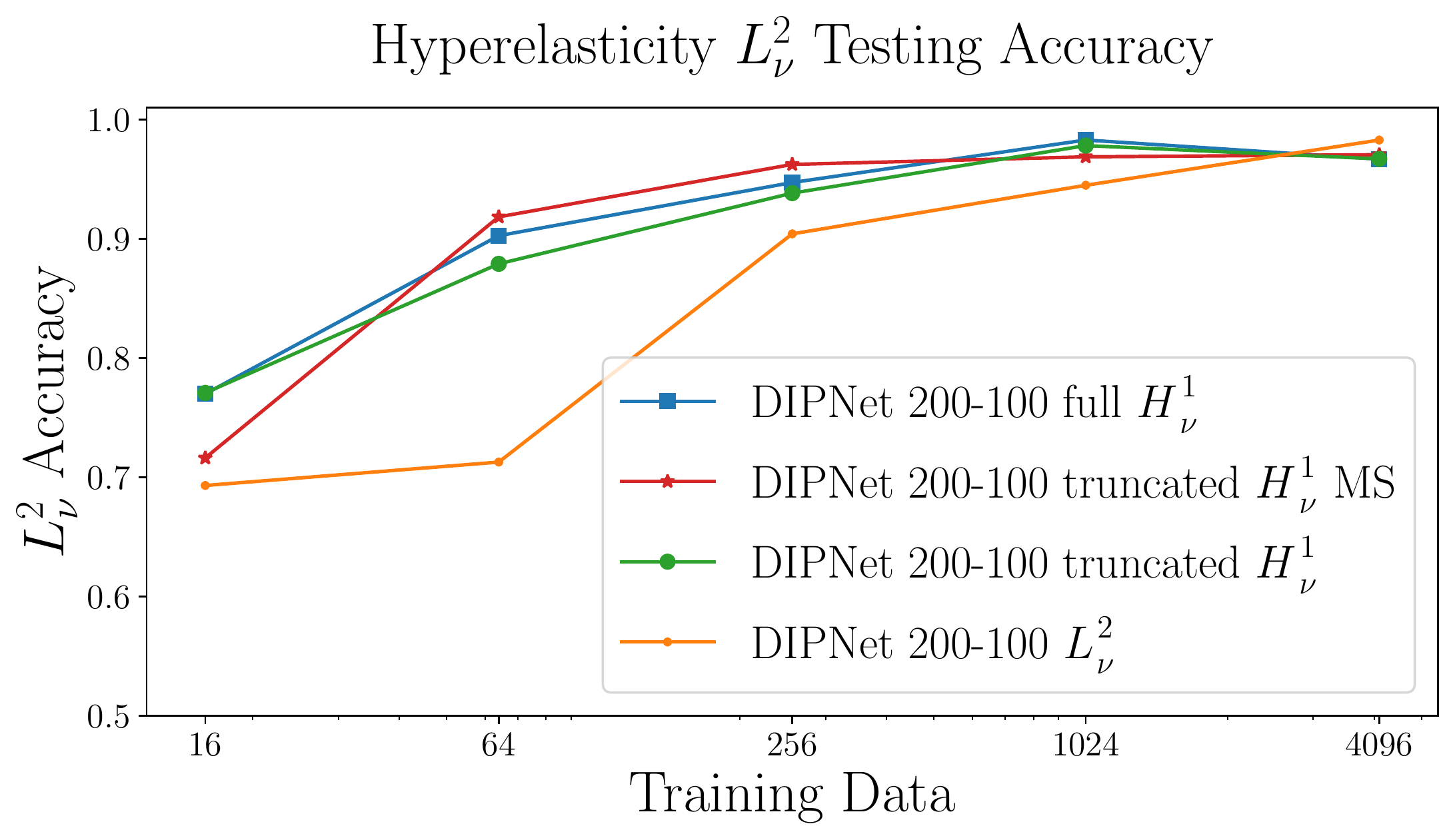}
\end{subfigure}%
\begin{subfigure}{0.5\textwidth}
\includegraphics[width = \textwidth]{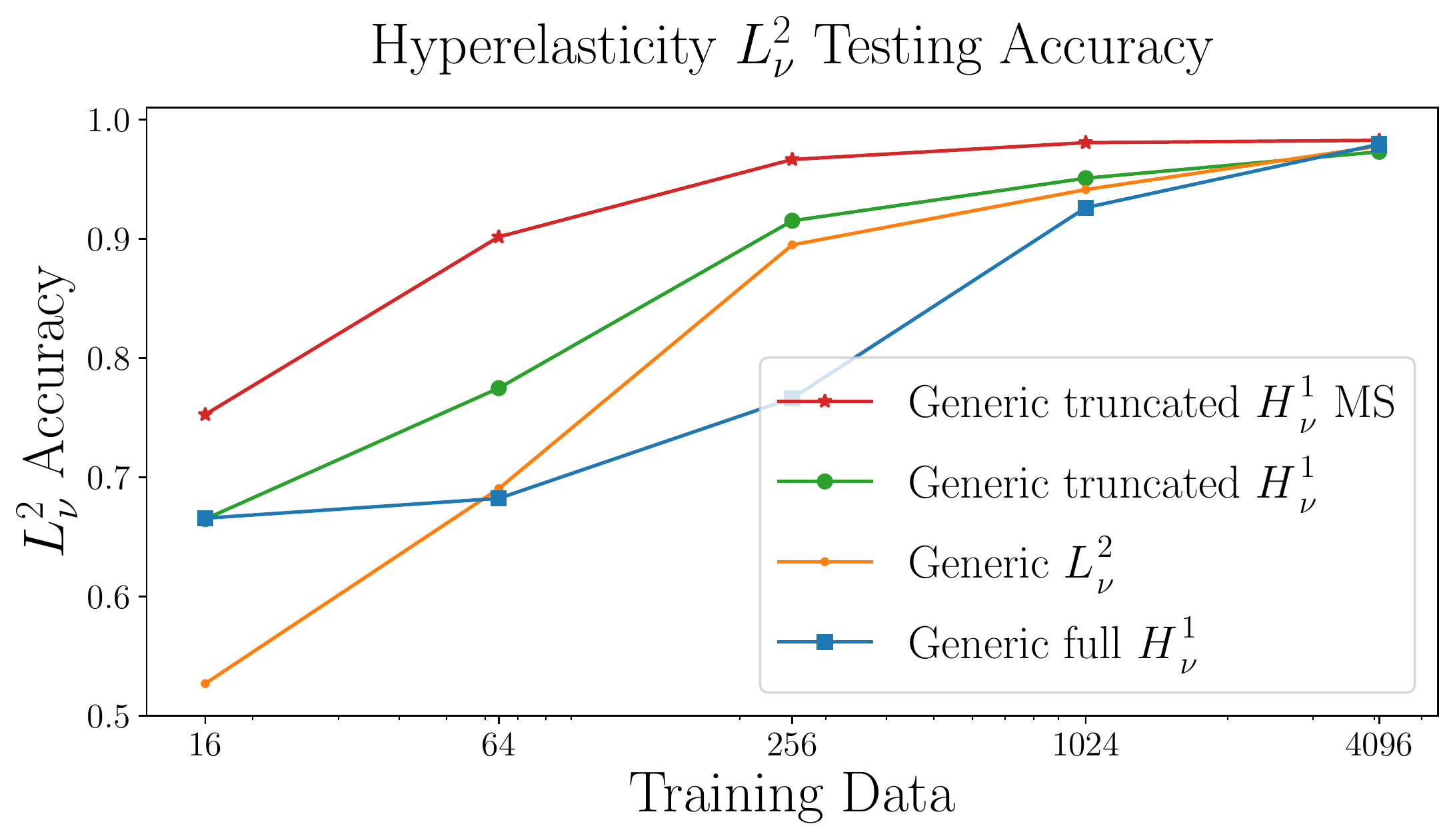}
\end{subfigure}
\caption{$L^2_\nu$ testing accuracies for reaction-diffusion (top row), convection-reaction diffusion (middle row) and the hyperelasticity problem (bottom row). The left column shows accuracies for the DIPNet architecture used in the corresponding problem, and the right column shows the generic encoder-decoder. The incorporation of derivative information improved function approximation in almost every case, except full $H^1_\nu$ training for the generic encoder-decoder in the limited data setting.}
\label{l2_comps}
\end{figure}

We note that the incorporation of the Jacobian information into the training loss function requires extra computation. However, as discussed in Section \ref{section:scalable_computation}, these computations are often negligible in comparison to the costs of evaluating the nonlinear forward map $m \mapsto q$. Thus these results make a strong case for the inclusion of derivative information in $L^2_\nu$ regression problems as an economical means of improving function approximation accuracy in the limited training data regime that is of practical interest, for nonlinear problems where Jacobian matrix-vector products are inexpensive in comparison to the evaluation of $m \mapsto q$.

\subsection{Evaluating the Derivative Accuracy}

We proceed by investigating derivative accuracy. Based on how the derivative learning problem is formulated, it is natural to first consider $H^1_\nu$ semi-norm accuracy, which is defined as:

\begin{equation}
  \text{$H^1_\nu$ semi-norm accuracy:} \quad  \left(1 - \sqrt{\mathbb{E}_\nu\left[\frac{\|\nabla q - \nabla f_w\|^2_{F(\mathbb{R}^{d_Q\times d_M})}}{\|\nabla q\|^2_{F(\mathbb{R}{^{d_Q\times d_M}})}}\right]}\right).
\end{equation}

We begin by noting that this is a truly exacting metric, since unlike the $L^2_\nu$ accuracy, which requires that only $d_Q$ entries of a vector be well-approximated across parameter space, here we require that $d_Q\times d_M$ entries of a (Jacobian) matrix be well-approximated across parameter space. For the first two problems this matrix has $211,250$ entries. For the hyperelasticity problem it has $845,000$ entries. To avoid repetition, in Figure \ref{hyperelasticity_h1_comps} we show results for the hyperelasticity problem, which are representative of the other two problems. The general takeaway is that the DIPNets with $H^1_\nu$ training performed the best in Jacobian predictions. For the generic encoder-decoder, derivative learning was substantially harder, and only the full $H^1_\nu$ training performed similar to the DIPNets ($\sim 70\%$ accuracy), and this was only the case when a significant amount of training data was available. Additionally, Jacobian accuracies can be extremely poor for generic encoder-decoder networks, which makes a strong case for not differentiating networks that have not been trained with derivative information.

\begin{figure}
\begin{subfigure}{0.5\textwidth}
\includegraphics[width = \textwidth]{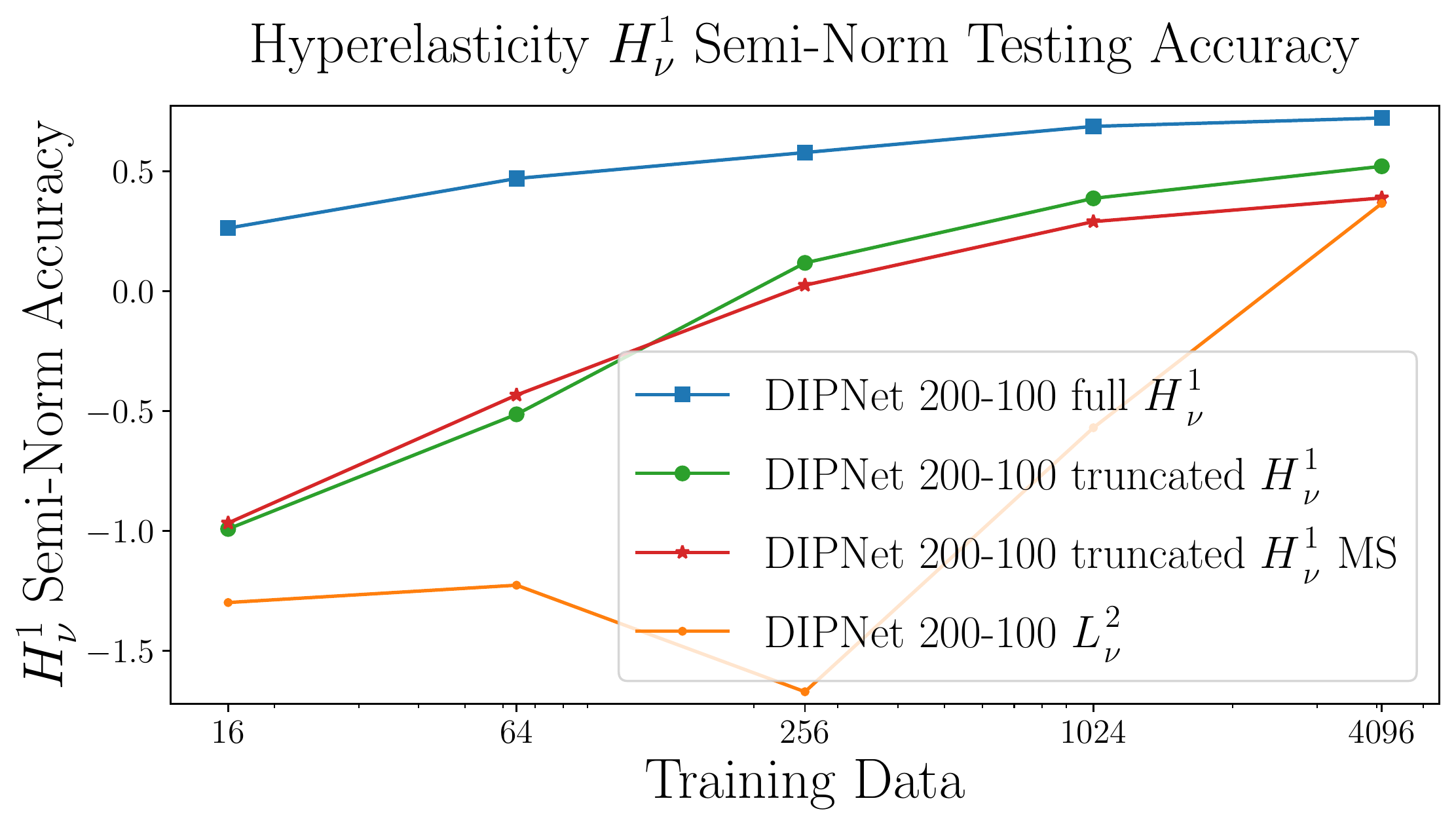}
\end{subfigure}%
\begin{subfigure}{\textwidth}
\includegraphics[width = 0.5\textwidth]{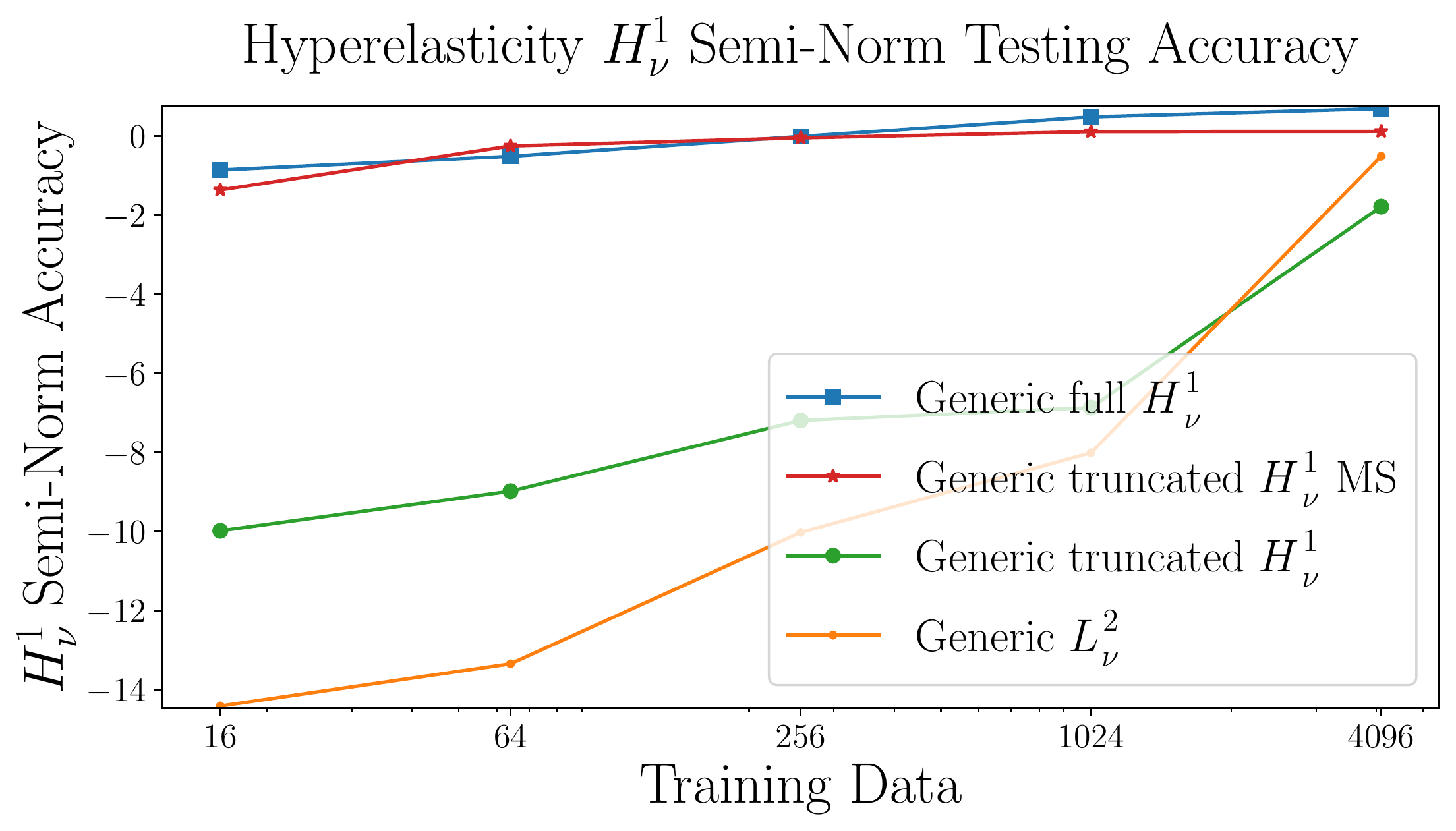}
\end{subfigure}
\caption{Hyperelasticity $H^1_\nu$ semi-norm accuracies for the DIPNet 200-100 (left), and a study of $H^1_\nu$ semi-norm accuracies for the Generic architecture (bottom). The left plot shows that the full $H^1_\nu$ training yielded the best Jacobian predictions for the DIPNet 200-100 architecture. The right plot demonstrates the unreliability of derivative approximations for generic-encoder architectures trained without derivative information. }
\label{hyperelasticity_h1_comps}
\end{figure}
\subsubsection{Gradients and (Gauss--Newton) Hessians of Inverse Problems}

As was previously mentioned, the Jacobian Frobenius metric is overly-exacting, and while it is directly related to the training formulations that we have proposed, is not a necessary indicator of the usefulness of derivative information in outer-loop contexts. In order to better assess the impact of derivative approximation errors, we propose to look at errors in gradient vectors and Gauss--Newton Hessian matrices, which are critical for the scalable solution of high-dimensional inverse problems. Gradient accuracy illuminates how both function and derivative accuracy can propagate to another quantity, while Gauss--Newton Hessians illuminate the usefulness of the Jacobian predictions alone. 

Given noise-corrupted observations $d\in \mathbb{R}^{d_Q}$ of the solution of the PDE (state) at the observable locations, we seek to find model parameters $m\in \mathbb{R}^{d_M}$ that are consistent with the model and the data. Given zero-mean noise, with covariance $\Gamma_\text{noise}$, the (deterministic) inverse problem can be formulated as 
\begin{equation}
  \min_m \Phi(m) :=  \underbrace{\frac{1}{2}\| q(m) - d\|_{\Gamma^{-1}_\text{noise}}^2}_\text{data misfit} + \mathcal{R}(m),
\end{equation}
where $\mathcal{R}$ is a regularization term. In the Bayesian inverse problem, the regularization term becomes a prior, which in the Gaussian case takes the form $\|m - m_0\|_{\mathcal{C}^{-1}}^2$, where $m_0$ is the parameter mean, and $\mathcal{C}$ is its covariance. The gradient of the misfit is $g = \nabla q(m)^T\Gamma^{-1}_\text{noise}(q(m) - d)$. We use the following accuracy metric for gradients:

\begin{equation}
  \text{Gradient Accuracy:} \quad \left(1 - \sqrt{\mathbb{E}_\nu\left[\frac{\|g_\text{true} - g_\text{prediction}\|^2_{\ell^2(\mathbb{R}^{d_M})}}{\|g_\text{true}\|^2_{\ell^2(\mathbb{R}^{d_M})}}\right]} \right).
\end{equation}

To avoid repetition, we show representative results from the convection-reaction-diffusion problem in Figure \ref{crd_g_comps}. We generate $d$ by evaluating the PDE map and corrupting the observables with $1\%$ Gaussian noise. In general the DIPNet with full $H^1_\nu$ training produced the most accurate gradients. The generic encoder-decoders overall did not produce gradients that were as accurate as those produced by the DIPNets; as in other cases the full $H^1_\nu$ training for the generic encoder-decoder struggled, particularly given limited training data. Interestingly the $L^2_\nu$ trained networks eventually produced reasonably accurate gradients. The gradient accuracy depends both on the function and the derivative accuracy, but it depends more mildly on the Jacobian accuracy than the function, as the Jacobian transpose action only needs to be accurate in the one-dimensional subspace spanned by the data misfit. In order to account for this, we average gradient errors over different data misfits, and locations in parameter space.

\begin{figure}
\begin{subfigure}{0.5\textwidth}
\includegraphics[width = \textwidth]{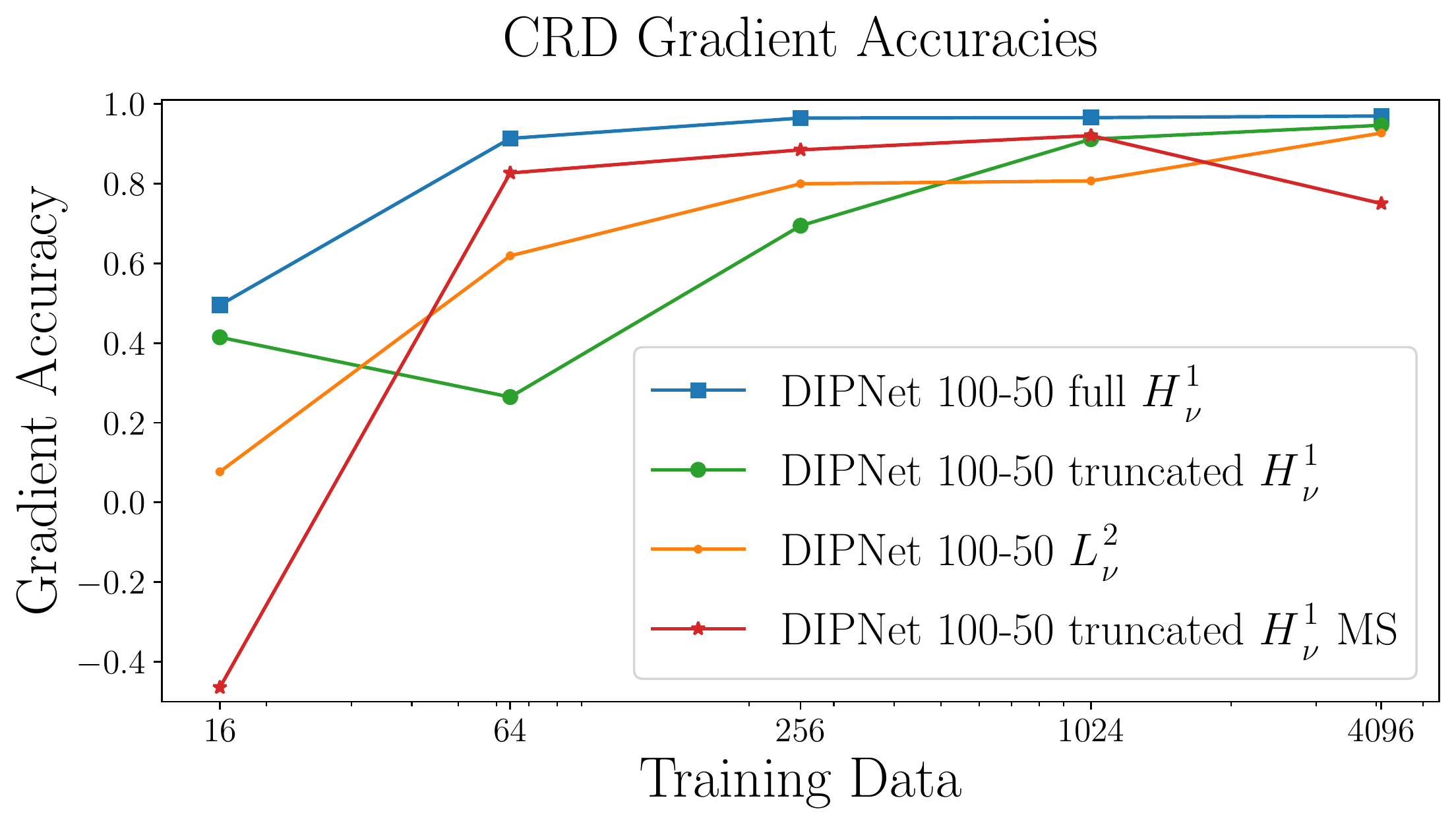}
\end{subfigure}%
\begin{subfigure}{0.5\textwidth}
\includegraphics[width = \textwidth]{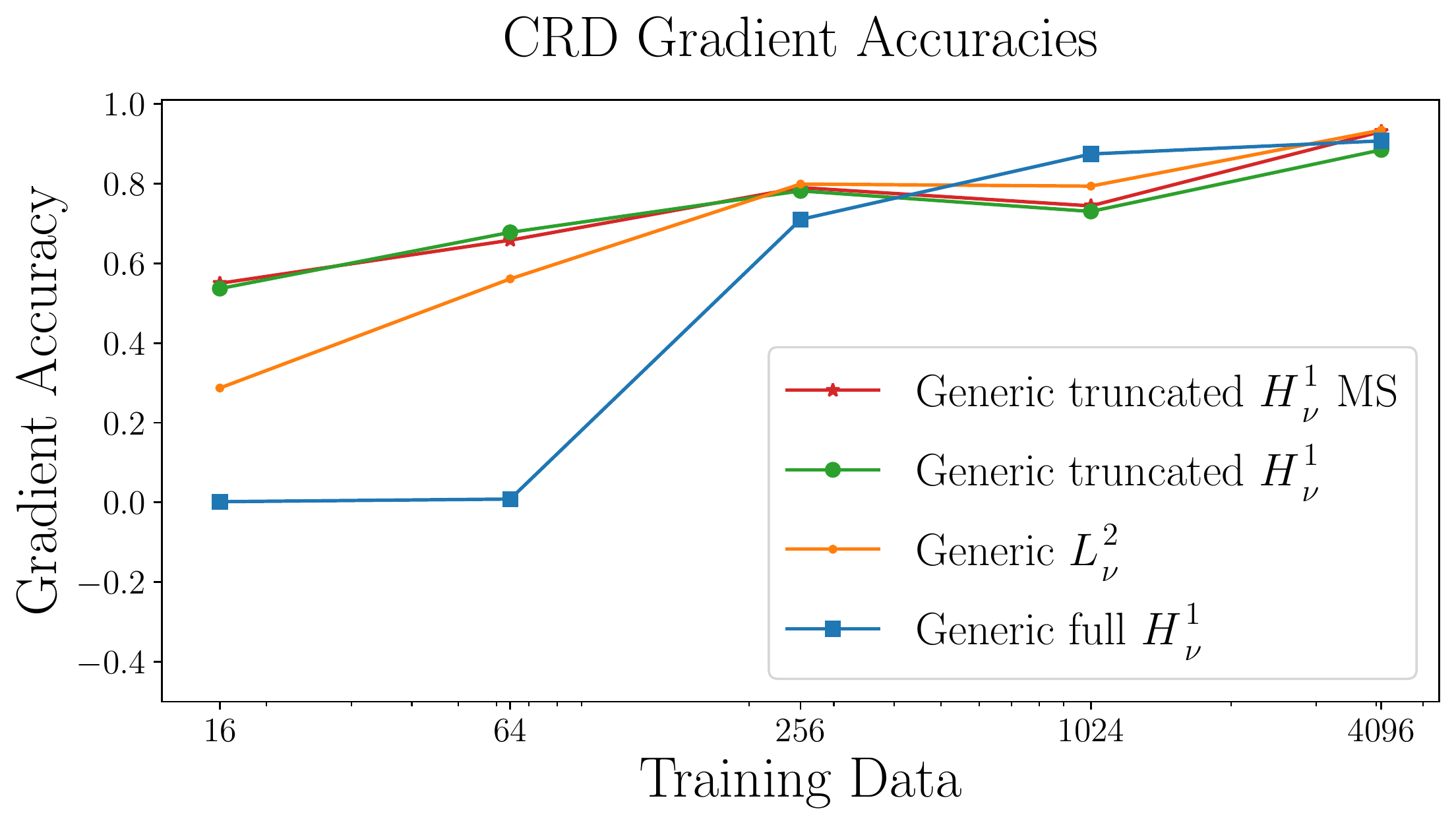}
\end{subfigure}
\caption{Inverse problem gradient accuracies for the DIPNet 100-50 (left) and generic encoder-decoder (right) for the CRD problem. The left plot demonstrates that the DIPNet can achieve highly accurate inverse problem gradients, in particular for the full $H^1_\nu$ training. In general the generic encoder-decoder network produced less accurate gradients than the DIPNet. Interestingly, in both cases the $L^2_\nu$ network eventually produced reasonably accurate gradients given enough training data.}
\label{crd_g_comps}
\end{figure}

Gauss--Newton Hessians, $\nabla q ^T \nabla q \in \mathbb{R}^{d_M\times d_M}$, are useful in many outer-loop problems. They can be used to accelerate the convergence rate of optimization methods (e.g., deterministic inverse problems), and generate proposals that yield higher acceptance rates and better mixing in Markov-chain Monte Carlo solutions of Bayesian inverse problems \cite{KimVillaParnoEtAl22}. They are useful in optimal experimental design problems, where their spectral information can be used to efficiently approximate the expected information gain. Being able to efficiently and accurately approximate Gauss--Newton Hessians over parameter space enables significant computational speedups in the aforementioned problems, as well as in other contexts. 

We study how well the trained neural operators can approximate the true Gauss--Newton Hessians across parameter space. We also study how well the neural operator--approximated Gauss--Newton Hessians agree with the true Gauss--Newton Hessians, in the span of the true Gauss--Newton Hessian's dominant eigenvectors (which coincide with the dominant Jacobian right singular vectors $V_r$). This allows us to better understand the nature of the Gauss--Newton Hessian approximation errors, since the discrepancy between the full and reduced Gauss--Newton accuracies represents errors made in the null-space of the true Gauss--Newton Hessian. For this task we propose the following two accuracy metrics:

\begin{subequations}
\begin{align}
&\text{Gauss--Newton accuracy:} \quad & \left(1 - \sqrt{\mathbb{E}_\nu\left[\frac{\|\nabla q^T\nabla q - \nabla f_w^T\nabla f_w\|^2_{F(\mathbb{R}^{d_M\times d_M})}}{\|\nabla q^T\nabla q\|^2_{F(\mathbb{R}{^{d_M\times d_M}})}}\right]}\right) \label{gn_accuracy}\\
&\begin{array}{r}
  \text{Reduced Gauss--Newton}\\
  \text{accuracy:}
  \end{array} \quad & \left(1 - \sqrt{\mathbb{E}_\nu\left[\frac{\|V_r^T(\nabla q^T\nabla q - \nabla f_w^T\nabla f_w)V_r\|^2_{F(\mathbb{R}^{r\times r})}}{\|V_r^T\nabla q^T\nabla qV_r\|^2_{F(\mathbb{R}{^{r\times r}})}}\right]}\right) \label{rgn_accuracy}
\end{align}
\end{subequations}

Figure \ref{gn_comps} shows the full Gauss--Newton Hessian accuracies \eqref{gn_accuracy} as well as reduced Gauss--Newton Hessian accuracies \eqref{rgn_accuracy} for all three problems. Poorly performing networks were omitted from these plots. A first important takeaway is that no $L^2_\nu$ trained networks were able to produce accurate Gauss--Newton Hessian approximations. The DIPNet full $H^1_\nu$ networks performed well in approximating the Gauss--Newton across parameter space across all three problems. The truncated $H^1_\nu$ training formulation was able to produce accurate approximations of the true Gauss--Newton Hessian in its dominant eigenvectors. However, due to the lack of nullspace information encoded in this training formulation, it was not able to accurately approximate the nullspace of the Gauss--Newton Hessians. In general the generic encoder-decoder networks were unable to produce accurate Gauss--Newton Hessian approximations. The exception however is in the hyperelasticity problem given large amounts of data. For this result we extended our training to $7168$, where the overparametrized generic encoder-decoder produced highly accurate $\sim 92.5\%$ accurate Gauss--Newton Hessian predictions. These results together tell an interesting story: the reduced basis DIPNet architectures can be trained efficiently and reliably generate accurate predictions of Gauss--Newton Hessians. However, given large amounts of training data, the reduced basis architecture accuracies stall, and the overparametrized network can produce superior approximations. This agrees with a classical heuristic in neural network training, that given more training data, and more representation power in the network, one can produce more accurate predictions.

\begin{figure}
\begin{subfigure}{0.4\textwidth}
\includegraphics[width = \textwidth]{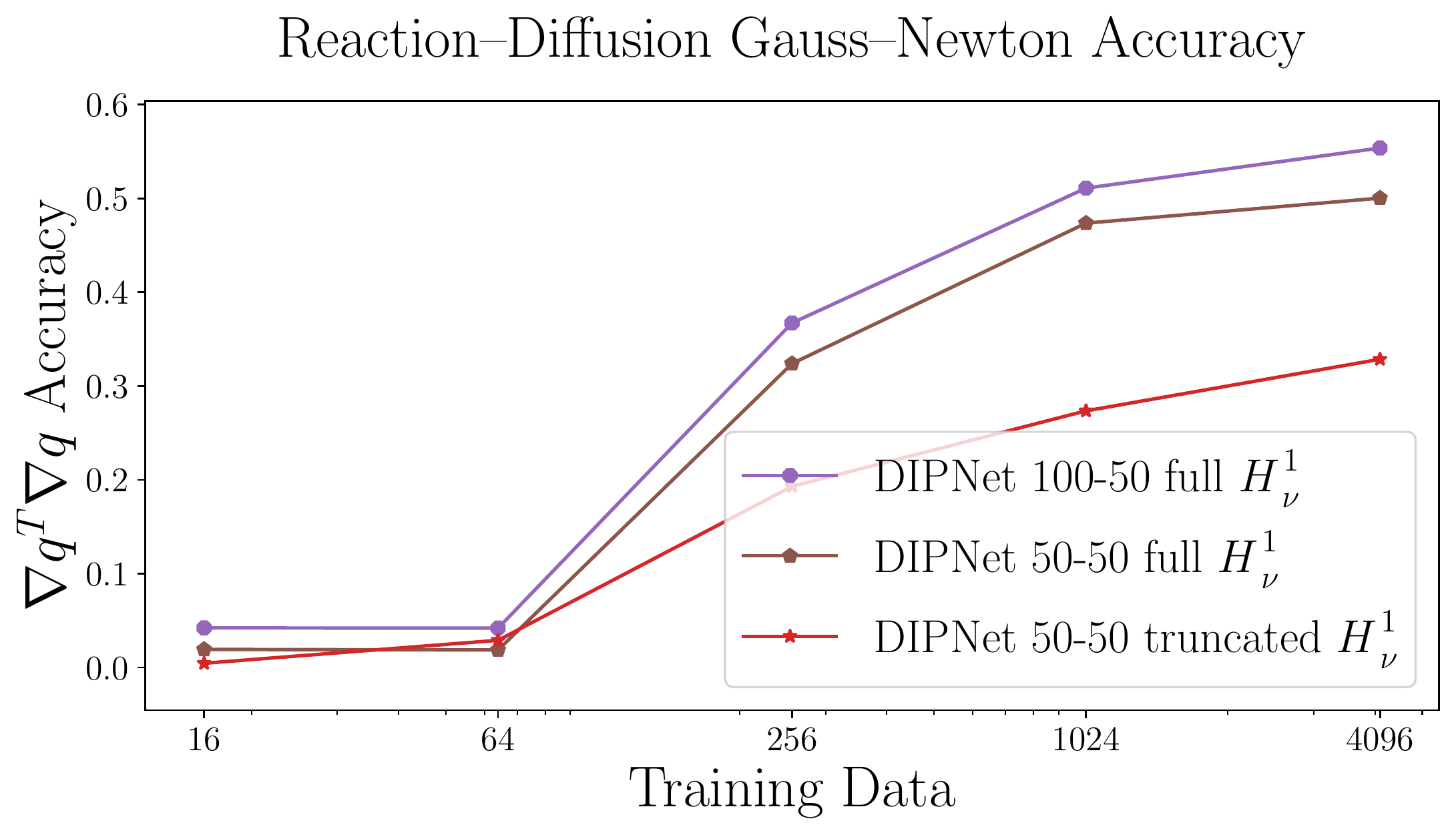}
\end{subfigure}%
\begin{subfigure}{0.6\textwidth}
\includegraphics[width = \textwidth]{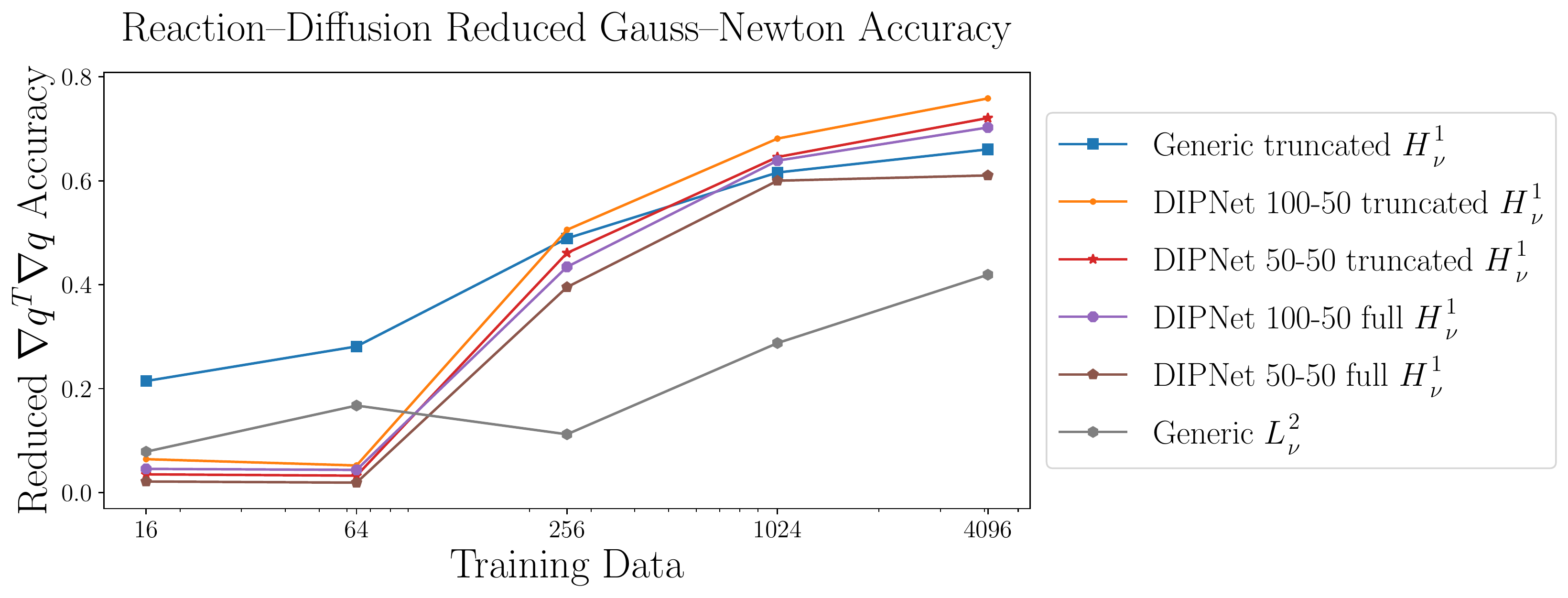}
\end{subfigure}
\begin{subfigure}{0.4\textwidth}
\includegraphics[width = \textwidth]{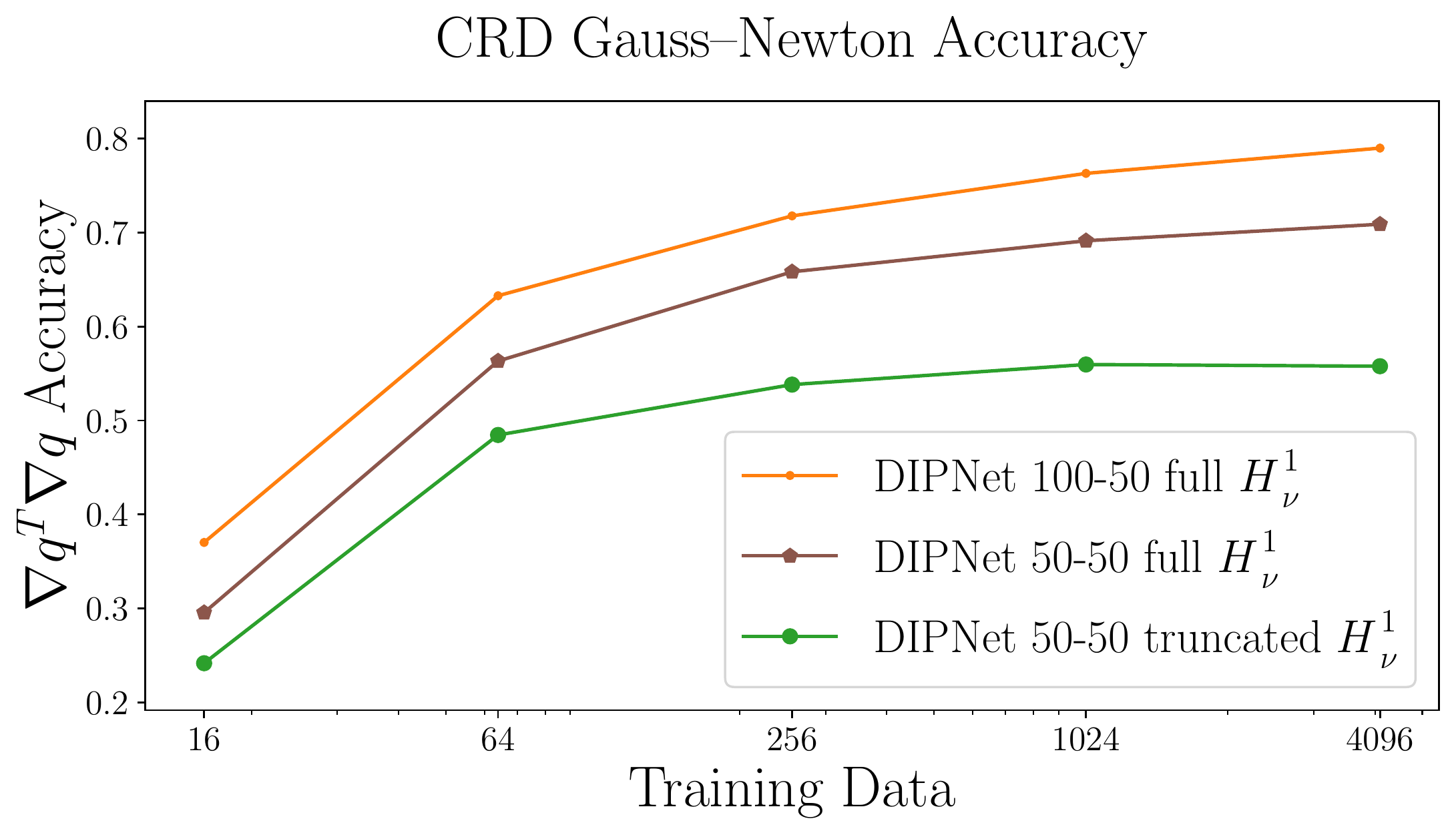}
\end{subfigure}%
\begin{subfigure}{0.6\textwidth}
\includegraphics[width = \textwidth]{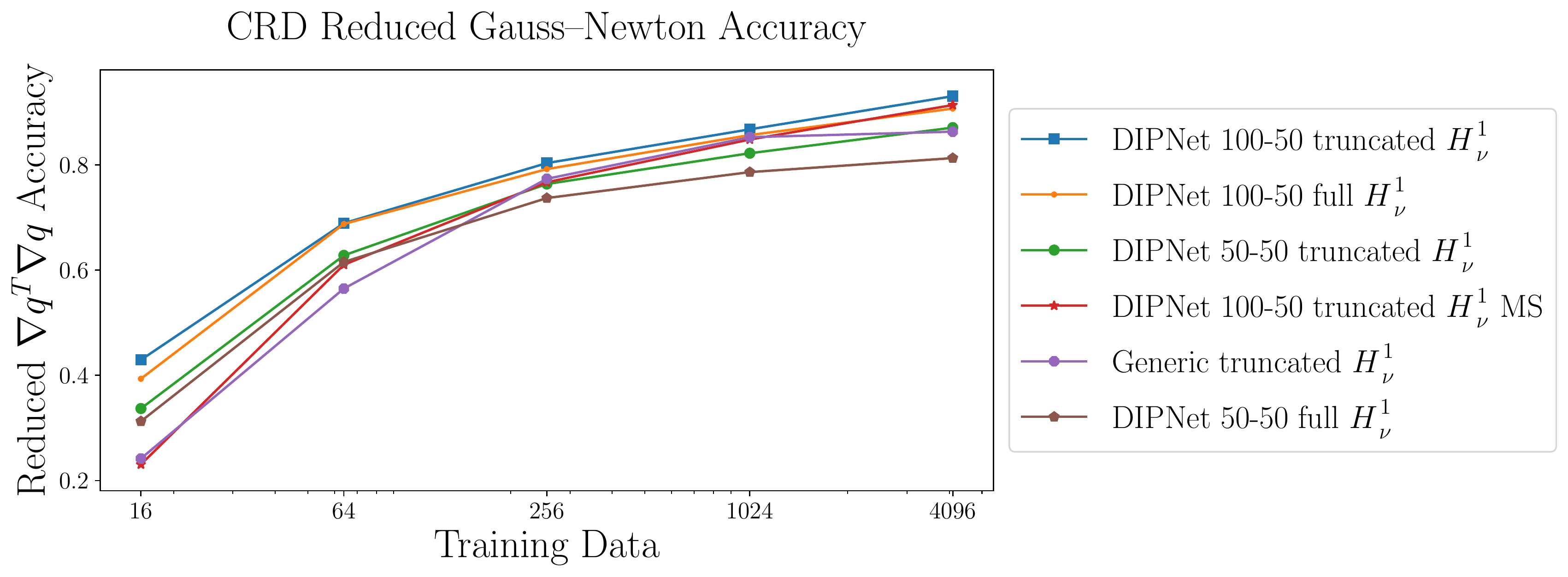}
\end{subfigure}
\begin{subfigure}{0.4\textwidth}
\includegraphics[width = \textwidth]{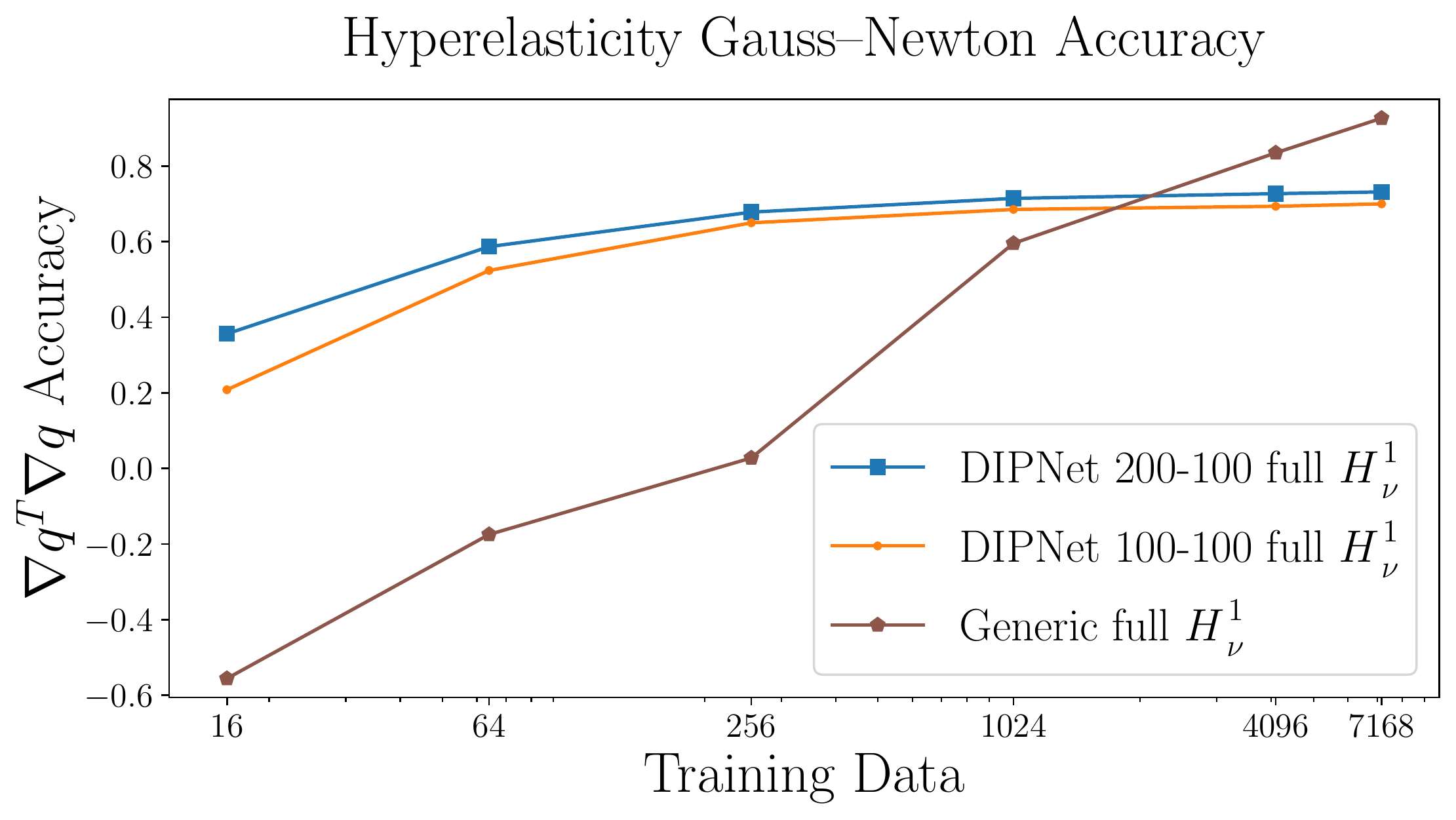}
\end{subfigure}%
\begin{subfigure}{0.6\textwidth}
\includegraphics[width = \textwidth]{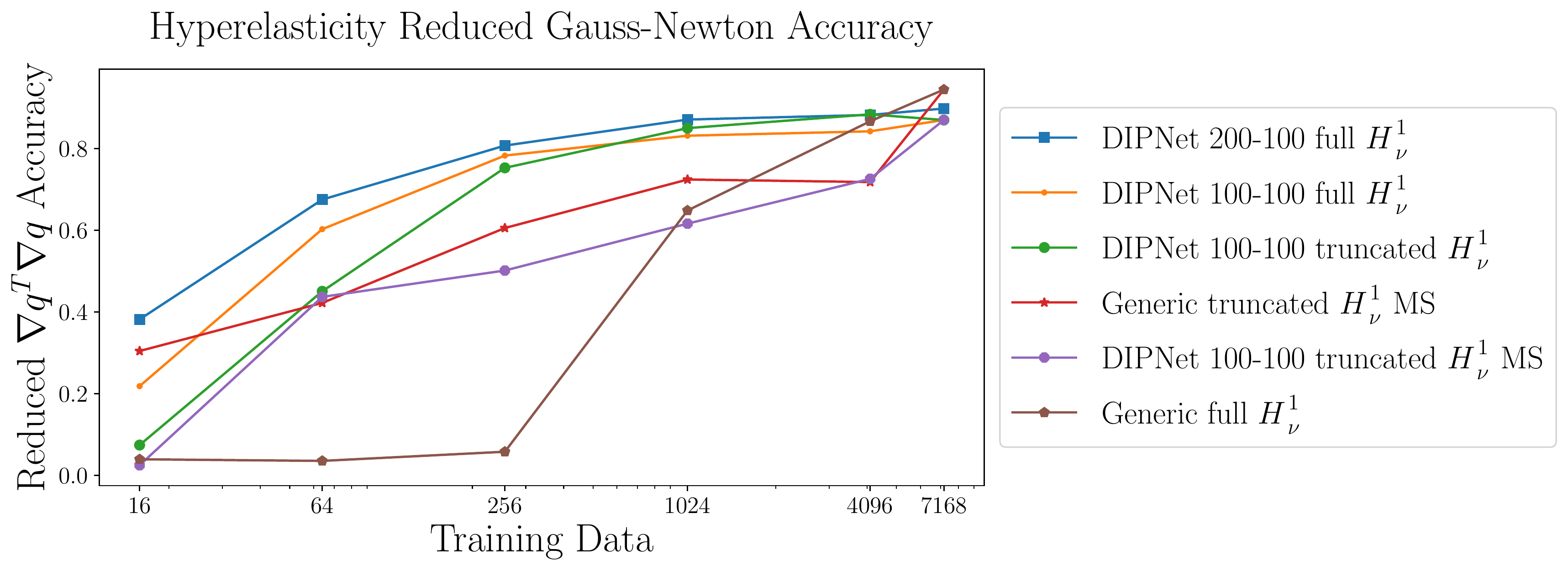}
\end{subfigure}
\caption{Gauss--Newton accuracies (left) and reduced Gauss--Newton accuracies (right) for the all three problems. The left plots shows that the full $H^1_\nu$ training yielded the best approximations of the Gauss--Newton Hessian. In particular, the DIPNet approximations were best for limited training data, but eventually the more expressive generic-encoder decoder outperformed for the hyperelasticity problem, given sufficient training data. We suspect that the stalling of the DIPNet accuracies (particularly evidenced in the hyperelasticity case) is due to the limited representation power of the chosen architecture. A deeper / wider network could lead to higher accuracy when the training set is richer. The right plot demonstrates again that the reduced Gauss--Newton predictions were better for the truncated $H^1_\nu$ training formulation, but these still failed to translate to accurate Gauss--Newton predictions due to the inability of this formulation to properly penalize the nullspace, which results in corrupted approximations of the true Gauss--Newton Hessian. }
\label{gn_comps}
\end{figure}

In summary the $H^1_\nu$ training with the reduced basis architecture provides an efficient and scalable means of learning high-dimensional derivatives, since it is a discretization dimension-independent method. This combination of network and training formulation led to improved function approximations and additionally superior derivative approximations to convention methods, in particular given limited training data. The generic encoder-decoder network in one case was able to produce superior predictions of Gauss--Newton Hessians to the DIPNets, given enough training data. However, this strategy is unfeasible as the size of the problem grows, since the method is dimension independent. The dimension-independent methods have the benefit of compute, energy and memory efficiency over dimension-dependent methods. Overall, numerical results demonstrate that these methods can be reliably deployed to help with both derivative and function accuracy for high-dimensional problems, if the dimensionality of the derivative information is independent of the discretization dimension of the problem. Derivative-informed neural operators can then be used to enable efficient derivative-based computations to accelerate the solution of outer-loop problems.

\section{Conclusions}
In this work, we investigated the difficulties of high-dimensional parametric derivative learning via neural operators, and proposed efficient and scalable methods that exploit low-dimensionality of derivative information, leading to practical dimension-independent methods for derivative learning. Low-dimensionality of derivative information can be proven analytically in many problems, and is demonstrated empirically in others.  We demonstrate that when derivative information is low-dimensional, both the generation of Jacobian training data and the derivative learning problem are computationally feasible via the use of efficient compression strategies that exploit this low-dimensionality. Compressed Jacobian training data can be computed matrix-free, and can be made computationally efficient by amortizing expensive computations used in the evaluation of the forward map. Compressed Jacobian information can then be encoded into several training formulations that leverage truncated SVD compression, or reduced dimensionality of a neural operator via a reduced basis architecture. 

In numerical results we explore the effects of these various strategies, which lead to the following conclusions:
\begin{enumerate}
	\item Derivative information can substantially improve the function approximation. In many cases it can be computed efficiently, by amortizing computations from the forward map, when the forward map is highly nonlinear, derivative training data can be used as an economical substitute for function training data, in order to achieve higher function accuracy.
	\item Neural operators with high-dimensional parameter dependence trained without derivative information yield inaccurate and unreliable predictions of derivative quantities such as Gauss--Newton Hessians. Therefore existing neural operators trained without derivative information are inadequate to be deployed in outer-loop problems that require derivative information to accelerate convergence (e.g., geometric MCMC \cite{BeskosGirolamiLanEtAl17,KimVillaParnoEtAl22,MartinWilcoxBursteddeetal2012}, or derivative-based variational transport \cite{ChenGhattas20,ChenWuChenEtAl19}).
	\item Reduced basis neural operator architectures are well-suited to derivative learning in very high-dimensional settings. As a result of Theorem \ref{theorem:dim_independent_dino}, derivative learning for these architectures is naturally dependent only on the reduced basis dimension, instead of the ambient discretization dimension. As the discretization dimension grows, the costs of derivative training become infeasible for dimension-dependent architectures, such as generic encoder-decoders, but remain constant for reduced basis architectures, such as DIPNets.

	\item Reduced basis architectures trained with derivative information led to significantly better function approximation. Given limited training data, these architectures often achieved $10-20\%$ better function approximations than those trained just on function data. These architectures yielded highly accurate ($>90\%$) inverse problem gradients; often $20-30\%$ more accurate than those trained without derivative information. Additionally, in no case did any method trained without derivative information give an accurate approximation of Gauss--Newton Hessians. Numerical results demonstrate that reduced basis networks can achieve good approximations of Gauss--Newton Hessians (e.g., $>70\%$ in the convection-reaction-diffusion problem) given limited training data. Generic encoder-decoders trained with significant training data were able to achieve good Gauss--Newton accuracies (e.g. $>90\%$ in the hyperelasticity problem). These results demonstrate that depending on the availability of training data, neural operators can be trained to efficiently approximate high-dimensional parametric maps and associated derivatives throughout parameter spaces.
\end{enumerate}

This work represents a significant step towards constructing neural operators that enable derivative-based outer-loop algorithms for efficiently exploring high-dimensional parameter spaces. In ongoing and future work, we plan to apply these derivative-informed neural operators to the solution of various outer-loop problems including Bayesian inverse problems, optimal experimental design, and optimization under uncertainty.

A fundamental limitation of this work is the assumption of the compressibility of the derivative information. While for large classes of problems one can prove analytically that derivative information is compact and discretization-independent, for certain problems including high frequency wave propagation and advection dominated flows, the derivative information may not be compressible via truncated SVD. New algorithmic innovations such as more sophisticated compression schemes will be required to enable efficient derivative learning in these contexts.

\section{Software Availability}
Code used in the numerical results is available in the \texttt{applications/} folder of \texttt{dino} \cite{dino} (\texttt{v0.2.0}), a python library that facilitates the construction of derivative-informed neural operators. Sampling of training data, and construction of reduced bases used in numerical results makes use of \texttt{hippyflow} \cite{hippyflow}. Adjoint-based derivative computations and parameter sampling  were performed using hIPPYlib \cite{VillaPetraGhattas2018,VillaPetraGhattas20}. FEniCS \cite{AlnaesBlechtaHakeEtAl2015} and PETSc \cite{BalayAbhyankarAdamsEtAl15} were used for finite element discretizations and solvers. Tensorflow was used for the machine learning \cite{AbadiAgarwalBarhamEtAl2016}.

\section*{Acknowledgements}
This research was supported by ARPA-E grant DE-AR0001208; DOE grants DE-SC0019303, DE-SC0023171 and DE-SC0021239;
NSF DMS grant 2012453; and DOD MURI FA9550-21-1-0084. The authors would like to thank Lianghao Cao and Dingcheng Luo for helpful comments during the preparation of work.

\addcontentsline{toc}{section}{References}
\bibliographystyle{IEEEtran}

\biboptions{sort,numbers,comma,compress}                 
\bibliography{local,references}

\section{Appendix} \label{section:appendix}

\subsection{Generation of Jacobian training data} \label{section:train_data_generation}

In this section we provide pseudo-code for the process of Jacobian training data generation, in the specific case of a steady-state PDE problem. The specific implementation we used for this work is contained in the \texttt{construct\_low\_rank\_Jacobians} method in the \texttt{ActiveSubspaceProjector} object in \texttt{hippyflow} \cite{hippyflow}.

\begin{algorithm}[H]
\caption{Jacobian training data generation for a steady-state PDE}\label{alg:data_gen}
\begin{algorithmic}
\State Given a reference distribution for the parameter, i.e., $m \sim \nu$.
\State Pre-decide appropriate global rank for Jacobian, $r$.
\State When quantity of interest is relatively low-dimensional, take $r = d_Q$.
\State Assemble function for $B = \frac{\partial q}{\partial u}$, i.e., the derivative of the quantity of interest with respect to the state. When the state is the quantity of interest, this is identity.

\For{$i =1,\dots,N$}
\State Sample $m_i \sim \nu$.
\State Solve state equation $R(u_i,m_i)=0$ for $u_i$.
\State Save output quantity of interest $q(m)$.
\If{using direct solver}
    \State Assemble linearized forward operator $\frac{\partial R}{\partial u}$ as a sparse matrix evaluated at $u_i, m_i$.
    \State Form sparse linear solver $\left[\frac{\partial R}{\partial u}\right]^{-1}$ via factorization (this includes the transpose solver).
\ElsIf{using Krylov solver}
    \State Assemble functions for the action and transpose action of $\frac{\partial R}{\partial u}$ on vectors, evaluated at $u_i, m_i$.
    \State Form Krylov linear solver $\left[\frac{\partial R}{\partial u}\right]^{-1}$ based on these functions.
\EndIf
\State Assemble a function for the action and transpose action of $\frac{\partial R}{\partial m}$ on a vector, evaluated at $u_i,m_i$.
\State Form the action and transpose action of $\nabla q = B\left[\frac{\partial R}{\partial u}\right]^{-1}\frac{\partial R}{\partial m}$ as a consequence of the previous steps.

\If{ using reduced basis neural operator with bases $\Phi_{\bar{r}_Q}$ and / or $\Psi_{\bar{r}_M}$}
    \If{$\mathbf{r_M}\leq\mathbf{r_Q}$}
        \State Compute right reduced Jacobian $\nabla q \Psi_{\bar{r}_M} $ matrix-free.

    \Else
        \State Compute left reduced Jacobian $\nabla q^T \Phi_{\bar{r}_Q}$ matrix-free.
    \EndIf
    \State Compute and save reduced Jacobian $\Phi_{\bar{r}_Q}^T\nabla q \Psi_{\bar{r}_M}$, or relevant left/right reduced Jacobian, if only output or input reduction is used in the neural operator.
\Else{ compute truncated SVD of Jacobian}
    \State $U_i,\Sigma_i, V_i$ = \texttt{randomizedSVD}($\nabla q$,\texttt{rank} = r) (see e.g., \cite{HalkoMartinssonTropp11,MartinssonTropp2020})
    \State Save $U_i,\Sigma_i,V_i$.

\EndIf

\EndFor
\end{algorithmic}
\end{algorithm}

The process of computing the training data for time-dependent quantity of interests is similar. For example, consider an observable at a final time, or at discrete measurements of a time-dependent system over a fixed window. In this case the action of $\left[\frac{\partial R}{\partial u}\right]^{-1}$ and its transpose remain as the linearized forward solve, and the adjoint solve. The former requires forward time-stepping, and the latter requires time-stepping in reverse starting from the final observation time. See e.g., \cite{BuiGhattasMartinetal2013,GhattasWillcox21,FarrellHamFunkeEtAl13,zhang2022petsc} for more in-depth discussions of efficient adjoint computations for time-dependent PDE problems.

\subsection{Schematics for derivative-informed neural operators}

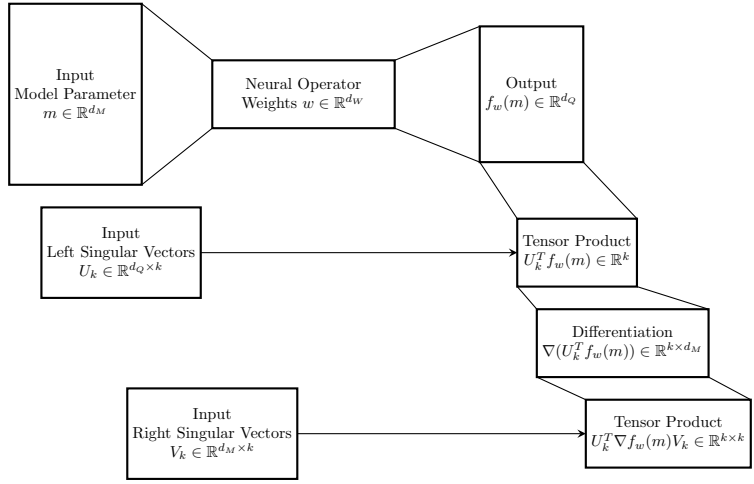
\begin{figure}[H]
\center
\begin{tikzpicture}[scale = 0.6, transform shape,every node/.style={draw,outer sep=0pt,thick}]
\node[bag] (Input) at (0,0) [minimum width=1cm,minimum height=4cm] {Input\\ Model Parameter\\
 $ m \in \mathbb{R}^{d_M}$};
 \node[bag] (InputUk) at (1,-3.5) [minimum width=1cm,minimum height=2cm] {Input\\ Left Singular Vectors\\
 $ U_k \in \mathbb{R}^{d_Q \times k}$};
  \node[bag] (InputVk) at (3,-7.5) [minimum width=1cm,minimum height=2cm] {Input\\ Right Singular Vectors\\
 $ V_k \in \mathbb{R}^{d_M \times k}$};

\node[bag] (NN) at (5,0) [minimum width=4cm,minimum height=1.5cm] {Neural Operator\\
                                                                    Weights $w \in \mathbb{R}^{d_W}$};
\draw (Input.north east) -- (NN.north west) (Input.south east) -- (NN.south west);
\node[bag] (Output) at (10,0)[minimum width=1cm,minimum height=3cm] {Output\\ $f_w(m) \in \mathbb{R}^{d_Q}$};
\draw (NN.north east) -- (Output.north west) (NN.south east) -- (Output.south west);

\node[bag] (fTU) at (11,-3.5)[minimum width=1cm,minimum height=1.5cm] {Tensor Product\\ $U_k^Tf_w(m) \in \mathbb{R}^{k}$};
\draw (Output.south west) -- (fTU.north west) (Output.south east) -- (fTU.north east);
\draw[-stealth] (InputUk.east) -- (fTU.west);
\node[bag] (DfTU) at (12,-5.5)[minimum width=1cm,minimum height=1.5cm] {Differentiation\\ $\nabla(U_k^Tf_w(m)) \in \mathbb{R}^{k\times d_M}$};
\draw (fTU.south west) -- (DfTU.north west) (fTU.south east) -- (DfTU.north east);
\node[bag] (DfTUV) at (13,-7.5)[minimum width=1cm,minimum height=1.5cm] {Tensor Product\\ $U_k^T\nabla f_w(m)V_k \in \mathbb{R}^{k\times k}$};
\draw(DfTU.south west) -- (DfTUV.north west) (DfTU.south east) -- (DfTUV.north east);
\draw[-stealth] (InputVk.east) -- (DfTUV.west);

\end{tikzpicture}
\caption{Schematic operation chart for extracting truncated Jacobian information from neural operators. This schematic is relevant for Propositions \ref{prop:decompose_h1_seminorm} and \ref{prop:submatrix_svd_convergence}. All operations can be vectorized over an additional data dimension.}\label{derivative_schematic}
\end{figure}

\begin{figure}[H]
\center
\begin{tikzpicture}[scale = 0.6, transform shape]


\node[bag] (Point) at (-5,0) [minimum width=1cm,minimum height=2cm] {Reduced Basis Neural Operator\\ in the Full Space};

\node[bag,draw,outer sep=0pt,thick] (Input) at (0,0) [minimum width=1cm,minimum height=4cm] {Input Reduced Basis\\ $\Psi_{\overline{r}_M} \in \mathbb{R}^{d_M \times \overline{r}_M}$\\
 $ \mathbb{R}^{d_M} \ni m \mapsto m_r \in \mathbb{R}^{\overline{r}_M}$};

\node[bag,draw,outer sep=0pt,thick] (NN) at (6.5,0) [minimum width=6.5cm,minimum height=1.5cm] {Reduced Space Neural Operator\\
                                                                    $\phi_{\overline{r}}(m_r,w) \in \mathbb{R}^{\overline{r}_Q}$\\
                                                                    Weights $w \in \mathbb{R}^{d_W}$};
\draw (Input.north east) -- (NN.north west) (Input.south east) -- (NN.south west);

\node[bag,draw,outer sep=0pt,thick] (Output) at (14,0)[minimum width=1cm,minimum height=3cm] {Output Reduced Basis\\ 
                                                                    $\Phi_{\overline{r}_Q} \in\mathbb{R}^{d_Q \times \overline{r}_Q}$\\
                                                                    Affine Shift $b \in \mathbb{R}^{d_Q}$\\
                                                                 $f_w(m) = \Phi_{\overline{r}_Q}\phi_{\overline{r}}(m_r,w) + b \in \mathbb{R}^{d_Q}$};

\draw (NN.north east) -- (Output.north west) (NN.south east) -- (Output.south west);

\node[bag] (Point1) at (-5,-5) [minimum width=1cm,minimum height=2cm] {\textbf{Train in Reduced Space}\\
                                                                        1. Choose $b$, e.g., $b = \texttt{sample\_mean}(q)$\\
                                                                        2. Reduce data:, $m_r = \Psi_{\overline{r}_M}^Tm$,\\
                                                                                          $\widehat{q}_r = \Phi_{\overline{r}_Q}^T(q - b)$\\
                                                                                           $\nabla \widehat{q}_r = \Phi_{\overline{r}_Q}^T\nabla q(m)\Psi_{\overline{r}_M}$ };

\node[bag,draw,outer sep=0pt,thick] (NN1) at (6.5,-5) [minimum width=6.5cm,minimum height=1.5cm] {3. Efficient Reduced Training Problem\\
                                                                    $\min_w \mathbb{E}_\nu[\|\phi_r - \widehat{q}_{\overline{r}}\|^2_2 + \|\nabla_{m_r}\phi_{\overline{r}} -  \nabla \widehat{q}_r\|_F^2]$ };

\draw[dotted] (NN.south west) -- (NN1.north west);
\draw[dotted] (NN.south east) -- (NN1.north east);

\end{tikzpicture}
\caption{Schematic for efficient training of reduced basis neural operators. This schematic is relevant to Theorem \ref{theorem:dim_independent_dino}.}\label{rb_dino_schematic}
\end{figure}
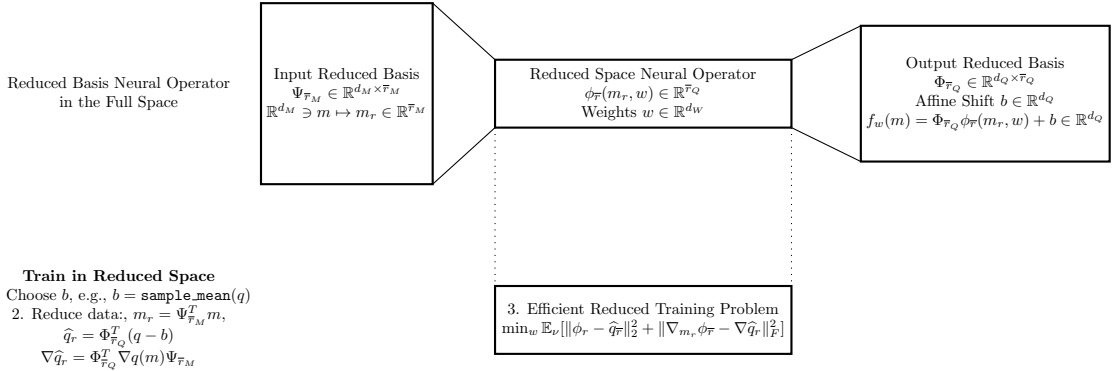

\subsection{Proof of Proposition \ref{prop:decompose_h1_seminorm}} \label{section:proof_of_h1_seminorm}

\begin{proof}
For clarity and completeness of exposition, we begin with two simple auxiliary identities. Let $A \in \mathbb{R}^{m\times n}$ be an arbitrary matrix and $Q_r \in \mathbb{R}^{n\times r}$ have orthonormal columns, (i.e.\ $Q_r^TQ_r = I_r \in \mathbb{R}^{r\times r}$). First, we have the orthogonal decomposition identity (the Pythagorean Theorem for Frobenius norm). We show this identity below as applied on the right side; it can be similarly derived on the left. We have

\begin{align}
\|A\|_F^2 = &\text{tr}(A^TA(I-Q_rQ_r^T))+ \text{tr}(A^TAQ_rQ_r^T) \nonumber \\
              = &\text{tr}((I-Q_rQ_r^T)^TA^TA(I-Q_rQ_r^T)) + \text{tr}(Q_rQ_r^TA^TAQ_rQ_r^T) \nonumber \\
                + &\text{tr}(Q_rQ_r^TA^TA(I-Q_rQ_r^T)^T) + \text{tr}((I-Q_rQ_r^T)^TA^TAQ_rQ_r^T)) \nonumber \\
              = & \|A(I - Q_rQ_r^T)\|_F^2 + \|AQ_rQ_r^T\|_F^2 \nonumber \\
               + &\text{tr}(Q_rQ_r^TA^TA(I-Q_rQ_r^T)^T) +\text{tr}((I-Q_rQ_r^T)^TA^TAQ_rQ_r^T)).
\end{align}
Unpacking the last line we have:
\begin{align}
&\text{tr}(\underbrace{Q_rQ_r^TA^TA}_{D^T}\underbrace{(I-Q_rQ_r^T)}_{E}) + \text{tr}(\underbrace{(I-Q_rQ_r^T)^T}_{E^T}\underbrace{A^TA Q_rQ_r^T}_{D}) \nonumber \\
= & 2 \text{tr}(A^TA(I-Q_rQ_r^T)Q_rQ_r^T) &(\text{since }\text{tr}(D^TE) = \text{tr}(E^TD) =\text{tr}(DE^T)) \nonumber \\
= & 0 &(\text{since }(I-Q_rQ_r^T)Q_rQ_r^T = 0),
\end{align}
so
\begin{equation}\label{orth_decomp_identity}
    \|A\|_F^2  = \|A(I - Q_rQ_r^T)\|_F^2 + \|AQ_rQ_r^T\|_F^2.
\end{equation}

Second, we have

\begin{align}\label{orth_reduction_identity}
    \|AQ_rQ_r^T \|^2_{F(\mathbb{R}^{m\times n})} = \text{tr}(AQ_rQ_r^TQ_rQ_r^TA^) = \text{tr}(AQ_rQ_r^TA^T) = \|AQ_r\|^2_{F(\mathbb{R}^{m\times r})}.
\end{align}

By repeated use of \eqref{orth_decomp_identity} identity we have
\begin{align}\label{eq:jac_error_decomposition}
    \|\nabla q - \nabla f_w\|^2_{F(\mathbb{R}^{d_Q \times d_M})} =  &\|U_rU_r^T(\nabla q - \nabla f_w)V_rV_r^T\|^2_{F(\mathbb{R}^{d_Q \times d_M})} \nonumber \\
    +&\|(I_{d_Q} - U_rU_r^T)(\nabla q - \nabla f_w) (I_{d_M}- V_rV_r^T)\|^2_{F(\mathbb{R}^{d_Q \times d_M})} \nonumber \\
     +&\|(I_{d_Q} - U_rU_r^T)(\nabla q - \nabla f_w)V_rV_r^T\|^2_{F(\mathbb{R}^{d_Q \times d_M})}\nonumber\\
     +&\|U_rU_r^T(\nabla q - \nabla f_w) (I_{d_M}- V_rV_r^T)\|^2_{F(\mathbb{R}^{d_Q \times d_M})}.
\end{align}
We can break up the second term of \eqref{orth_decomp_identity} using the triangle inequality:
\begin{align}
&\|(I_{d_Q} - U_rU_r^T)(\nabla q - \nabla f_w) (I_{d_M}- V_rV_r^T)\|^2_{F(\mathbb{R}^{d_Q \times d_M})} \leq \nonumber \\
&\|(I_{d_Q} - U_rU_r^T)\nabla q (I_{d_M}- V_rV_r^T)\|^2_{F(\mathbb{R}^{d_Q \times d_M})} + \|(I_{d_Q} - U_rU_r^T)\nabla f_w (I_{d_M}- V_rV_r^T)\|^2_{F(\mathbb{R}^{d_Q \times d_M})}
\end{align}

Note next that we have
\begin{subequations} \label{eq:j_decomp_identities}
\begin{align}
    (I_{d_Q} - U_rU_r^T)\nabla q (I_{d_M}- V_rV_r^T) &= \nabla q - U_r \Sigma_r V_r^T\\
    (I_{d_Q} - U_rU_r^T)\nabla q V_rV_r^T &= 0 \\
    U_rU_r^T\nabla q (I_{d_M}- V_rV_r^T) &= 0.
\end{align}
\end{subequations}
Applying these formulae along with \ref{eq:qSVD} and several applications of \eqref{orth_reduction_identity} we get the desired result:

\begin{align}
    \|\nabla q - \nabla f_w\|^2_{F(\mathbb{R}^{d_Q \times d_M})} \leq  &\|\Sigma_r - U_r^T\nabla f_wV_r\|^2_{F(\mathbb{R}^{r \times r})} \nonumber + \sum_{i=r+1}^{\min\{d_Q,d_M\}}\sigma_i^2 \\
    +&\| (I_{d_Q} - U_rU_r^T)\nabla f_w (I_{d_M}- V_rV_r^T)\|^2_{F(\mathbb{R}^{d_Q \times d_M})} \nonumber \\
     +&\|(I_{d_Q} - U_rU_r^T)\nabla f_wV_r\|^2_{F(\mathbb{R}^{d_Q \times r})}\nonumber\\
     +&\|U_r^T\nabla f_w (I_{d_M}- V_rV_r^T)\|^2_{F(\mathbb{R}^{r \times d_M})}.
\end{align}

\end{proof}

\subsection{Proof of Proposition \ref{prop:submatrix_svd_convergence}} \label{section:proof_of_subsampled_svd}

\begin{proof}
Let $B = U_r\Sigma_rV_r^T - \nabla f_w$. We begin with the case of independent left and right column samples. Let $\chi_{i\in [\widetilde{k}]}$ denote the indicator function that the index $i$ is in the index subset $[\widetilde{k}]$, which implies $\mathbb{E}_{[\widetilde{k}] \sim \kappa} [\chi_{i\in [\widetilde{k}]}] = \frac{k}{r}$ for any $i = 1, \dots, r$. We have then that
\begin{align}
    \|U_{[\widehat{k}]}^TBV_{[\widetilde{k}]}\|_{F(k\times k)}^2 = \sum_{i=1}^{r}\sum_{j=1}^r \chi_{i \in [\widehat{k}]}\chi_{j \in [\widetilde{k}]}(U_r^TBV_r)_{ij}^2,
\end{align}
where we used $\chi_{i\in[\widetilde{k}]} = \chi_{i\in[\widetilde{k}]}^2$. Taking expectation and using the independence of the samples, we have
\begin{align}
    \mathbb{E}_{[\widehat{k}],[\widetilde{k}] \sim \kappa}\left[\|U_{[\widehat{k}]}BV_{[\widetilde{k}]})\|_{F(k\times k)}^2\right] = \mathbb{E}_{[\widehat{k}],[\widetilde{k}] \sim \kappa}\left[\sum_{i=1}^{r}\sum_{j=1}^r \chi_{i \in [\widehat{k}]}\chi_{j \in [\widetilde{k}]}(U_r^TBV_r)_{ij}^2 \right] \nonumber \\ 
    = \sum_{i=1}^{r}\sum_{j=1}^r\mathbb{E}_{[\widehat{k}],[\widetilde{k}] \sim \kappa}\left[ \chi_{i \in [\widehat{k}]}\chi_{j \in [\widetilde{k}]}\right](U_r^TBV_r)_{ij}^2   = \frac{k^2}{r^2}\sum_{i=1}^{r}\sum_{j=1}^r(U_r^TBV_r)_{ij}^2.
\end{align}
For the case that the left and right samples are the same we have
\begin{align}
    \|U_{[\widehat{k}]}^TBV_{[\widehat{k}]}\|_{F(k\times k)}^2 &= \sum_{i=1}^{r}\sum_{j=1}^r \chi_{i \in [\widehat{k}]}\chi_{j \in [\widehat{k}]}(U_r^TBV_r)_{ij}^2 \nonumber \\
    &= \sum_{i=1}^{r} \chi_{i \in [\widehat{k}]}(U_r^TBV_r)_{ii}^2 +\sum_{i=1}^{r}\sum_{i \neq j=1}^r \chi_{i \in [\widehat{k}]}\chi_{j \in [\widehat{k}], j\neq i}(U_r^TBV_r)_{ij}^2 \nonumber. \\
\end{align}
Taking expectation we have

\begin{align}
    \mathbb{E}_{[\widetilde{k}]\sim \kappa}\left[\|U^T_{[\widehat{k}]}BV_{[\widehat{k}]})\|_{F(k\times k)}^2\right] = \nonumber
    \frac{k}{r}\sum_{i=1} (U_r^TBV_r)_{ii}^2 + \frac{k}{r}\frac{(k-1)}{(r-1)}\sum_{i=1}^{r}\sum_{i \neq j=1}^r (U_r^TBV_r)_{ij}^2.
\end{align}

\end{proof}

\end{document}